\documentclass[12pt]{amsart}
\usepackage{amsmath}
\usepackage{fourier}
\usepackage{amssymb}
\usepackage{amscd}
\usepackage{amsthm}
\usepackage[centertags]{amsmath}
\usepackage{amsfonts}
\usepackage{newlfont}
\usepackage{graphicx}
\usepackage[usenames]{color}
\usepackage{amsfonts, amssymb}
\usepackage{mathrsfs}
\usepackage{latexsym}
\usepackage{tikz}
\usepackage{verbatim}
\usepackage{tabulary,tabularx}
\usepackage[all]{xy}
\usepackage{enumitem}

\usepackage[colorlinks=true,linkcolor=colorref,citecolor=colorcita,urlcolor=colorweb]{hyperref}
\definecolor{colorcita}{RGB}{21,86,130}
\definecolor{colorref}{RGB}{5,10,177}
\definecolor{colorweb}{RGB}{177,6,38}

\usepackage[latin1]{inputenc}

\newtheorem{theorem}{Theorem}[section]

\newtheorem{proposition}[theorem]{Proposition}
\newtheorem{corollary}[theorem]{Corollary}
\newtheorem{lemma}[theorem]{Lemma}

\theoremstyle{definition}
\newtheorem{remark}[theorem]{Remark}

\theoremstyle{definition}

\theoremstyle{remark}

\newcommand{\dis}{\displaystyle}

\newcommand{\Jj}{\mathcal J}

\newcommand{\NN}{\mathbb N}
\newcommand{\zN}{\mathbb N}

\newcommand{\CC}{\mathbb C}

\newcommand{\RR}{\mathbb R}
\newcommand{\jj}{\mathbf{j}}
\newcommand{\ii}{\mathbf{i}}

\usepackage{mathtools}

\DeclareMathOperator{\mon}{mon}
\DeclareMathOperator{\card}{card}

\DeclareMathOperator{\id}{id}

\newcommand{\veps}{\varepsilon}

\newcommand{\Aa}{\mathcal{A}}

\newcommand{\Pp}{\mathcal{P}}
\newcommand{\Rr}{\mathcal{R}}
\newcommand{\Ff}{\mathcal{F}}

\newcommand\restrict[1]{\raisebox{-.5ex}{$\vert$}_{#1}}

\begin{document}
	\title{Monomial convergence on $\ell_r$}
	
	\author{Daniel Galicer \and Mart\'in Mansilla \and Santiago Muro \and Pablo Sevilla-Peris}

\thanks{This work was partially supported by projects CONICET PIP 11220130100329CO,   ANPCyT PICT 2015-2224,  ANPCyT PICT 2015-2299, ANPCyT PICT 2015 - 3085, UBACyT 20020130100474BA, UBACyT 20020130300052BA, UBACyT 20020130300057BA. The second author was supported by a CONICET doctoral fellowship. The fourth author was supported by \textsc{mineco} and \textsc{feder} Project MTM2017-83262-C2-1-P}

\begin{abstract}
For $1 < r \le 2$, we study the set of monomial convergence for spaces of holomorphic functions over $\ell_r$. For $ H_b(\ell_r)$, the space of  entire functions of bounded type in $\ell_r$,  we prove that $\mon H_b(\ell_r)$ is exactly the Marcinkiewicz sequence space $m_{\Psi_r}$ where the symbol $\Psi_r$ is given by $\Psi_r(n) := \log(n + 1)^{1 - \frac{1}{r}}$
for $n \in \NN_0$. 
For the space of $m$-homogeneous polynomials on $\ell_r$, we prove that the set of monomial convergence $\mon  \Pp (^m \ell_r)$ contains the sequence space $\ell_{q}$ where $q=(mr')'$. Moreover, we show that for any $q\leq s<\infty$, the Lorentz sequence space $\ell_{q,s}$ lies in $\mon  \Pp (^m \ell_r)$, provided that $m$ is large enough.
We apply our results to make an advance in the description of the  set of monomial convergence of $H_{\infty}(B_{\ell_r})$ (the space of bounded holomorphic on the unit ball of $\ell_r$). As a byproduct we close the gap on certain estimates related with the \emph{mixed} unconditionality constant for spaces of polynomials over classical sequence spaces.

\end{abstract}

\maketitle

\section{Introduction and main results}

A basic fact taught on every course of one complex variable is that every function that is differentiable at all points of a disc centred at $0$ can be represented as a power series, and vice-versa. In other words, the derivative $f'(z)$ exists (i.e. $f$ is differentiable at $z$) for every $\vert z \vert < r$ if and only if there is a sequence of coefficients $(c_{n}(f))_{n} \subseteq \mathbb{C}$ so that
\begin{equation} \label{franz}
f(z) = \sum_{n=0}^{\infty} c_{n}(f) z^{n}
\end{equation}
for every $\vert z \vert < r$ (i.e. it is analytic). In this case the coefficients can be computed either by differentiation or by the Cauchy integral formula, and the convergence is absolute and uniform on each compact subset of the disc. It also 
rather elementary to see that in fact this extends also to functions on several complex variables: a function defined on a Reinhardt domain $\mathcal{R} \subset \mathbb{C}^{n}$ (all needed definitions in this introduction can be found in Section~\ref{preliminares}), is differentiable at every $z \in \mathcal{R}$ if and only it is analytic (and has a power series representation as in \eqref{franz}). So, differentiability and analiticity are two equivalent ways to define holomorphy in one and several complex variables. \\

The idea of developing a sort of function theory in infinitely many variables (or, to put in nowadays terms, on infinite dimensional spaces) started at the beginning of the 20th century with the work, among others, 
of Hilbert, Fr\'echet and G\^ateaux. Here the problem becomes much more subtle. To begin with, while a notion such as differentiability can be considered for functions on any Banach space the idea of analiticiy, 
where one needs power expansions  with monomials of the form $z^{\alpha} = z_{1}^{\alpha_{1}} \cdots z_{n}^{\alpha_{n}}$, is much more restrictive. A Schauder basis, where an idea of 
`coordinate' makes sense, is at least needed. This shows that the approach to holomorphy through differentiability is much more far reaching than the one through analiticity. We say, then, that a
function $f : U \to \mathbb{C}$ (where $U$ is some open subset of a Banach space $X$) is holomorphic if it is Fr\'echet differentiable at every point of $U$ (or, equivalently, continuous and holomorphic when restricted to any one-dimensional affine subspace, see \cite{mujica1986,dineen1999complex}).\\
It is also worthy to explore the analytic approach whenever it makes sense  (as, for example Banach sequences spaces, the definition is given below). Let us succinctly explain how this works (a detailed account on this can be found in \cite[Chapter~15]{defant2019libro}). Let $f$ be a 
holomorphic function on some Reinhardt domain $\mathcal{R}$ in a Banach sequence space $X$. For each fixed $n$, the restriction of $f$ to $\mathcal{R}_{n} = \mathcal{R} \cap \mathbb{C}^{n}$ (which 
is a Reinhardt domain) is holomorphic and, therefore, has a monomial expansion with coefficients $(c_{\alpha}^{(n)}(f))_{\alpha  \in \mathbb{N}_{0}^{n}}$. It is easy to check that 
$c_\alpha^{(n)}=c_\alpha^{(n+1)}$ for $\alpha\in\mathbb N_0^n\subset \mathbb N_0^{n+1}$. In other words, we have a  a unique family $(c_\alpha(f))_{\alpha\in\mathbb N_0^{(\mathbb N)}}$, such that
\begin{equation} \label{savall}
f(z)=\sum_{\alpha\in\mathbb N_0^{(\mathbb N)}}c_\alpha z^\alpha
\end{equation}
for all $n\in\mathbb N$ and all $z\in \mathcal{R}_n$. The coefficients can be computed, for each $\alpha = (\alpha_{1} , \ldots , \alpha_{n},0,0, \ldots)$, by 
\begin{equation} \label{dowland}
     c_{\alpha}(f) = \frac{\partial^{\alpha}f(0)}{\alpha !} = \frac{1}{(2\pi i)^n} \int_{ \lbrace \vert z \vert = r \rbrace} \frac{f(z)}{z^{\alpha + 1}} dz,
\end{equation}
where $r>0$ such that $ \lbrace \vert z \vert \leq r \rbrace \subset \mathcal R$.
As usual, the power series $\sum_\alpha c_\alpha z^\alpha$ is called the monomial expansion of $f$.\\
One could expect that in the settings where these two approaches coexist  they are equivalent, just as in the finite dimensional setting. But this is not the case. When dealing with a totally different problem, related to the 
convergence of Dirichlet series, Toeplitz  gave in \cite{toeplitz1913ueber} an example that, to what we are concerned here, provided a holomorphic function on $c_{0}$ and a point in $c_{0}$ for which the monomial 
expansion does not converge absolutely. This shows that there are holomorphic functions that are not analytic (the converse, however, holds true: every analytic function is holomorphic).\\
Then the question arises in a natural way: for which $z$'s  does the monomial expansion of every holomorphic function converge absolutely? (note that when this is the case when the series converges to $f(z)$). 
From \eqref{savall} we have that this happens for every $z \in \mathcal{R}_{n}$ but, are there other ones? Ryan showed in \cite{ryan1980applications} that the monomial expansion of every 
holomorphic function on $\ell_{1}$ converges at every $z \in \ell_{1}$. Later Lempert in \cite{lempert1999dolbeault} proved that the monomial expansion of every holomorphic function on 
$\rho B_{\ell_{1}}$ (for $\rho >0$) converges at every $z \in \rho B_{\ell_{1}}$. This is a somewhat extremal case, where the analytic and differential approaches coincide. What happens in other spaces? or if 
we consider smaller families of holomorphic functions? To tackle this questions the set of monomial convergence of a family $\mathcal F(\mathcal{R})$ of holomorphic functions on $\mathcal{R}$ was defined in 
\cite{defant2009domains} as
\[
    \mon \mathcal F(\mathcal{R})=\Bigg\{z\in \mathbb{C}^{\mathbb{N}}  \,\,\, \colon \,\,\, \sum_{\alpha\in\mathbb N_0^{(\mathbb N)} } \big|c_\alpha(f) z^\alpha \big| < \infty \,\, \text{ for all } f \in \mathcal F(\mathcal{R})  \Bigg\} \,,
\]
and a systematic study was started. We are mostly interested in studying the set of monomial convergence of the following three families:
\begin{itemize}
\item $H_b(\ell_r)$ (the space of holomorphic functions of bounded type on $\ell_r$)

\item $H_\infty(B_{\ell_r})$ (the space of bounded holomorphic functions on the open unit ball of $\ell_r$)

\item $\mathcal P(^m\ell_r)$ (the space of $m$-homogeneous polynomials on $\ell_r$).
\end{itemize}
The results of Ryan and Lempert mentioned before imply $\mon H_b(\ell_1) = \mon \mathcal P(^m\ell_1) = \ell_{1}$ for every $m$ and $\mon H_\infty(B_{\ell_1}) = B_{\ell_{1}}$. On the other endpoint of the 
scale ($p=\infty$)  \cite{bayart2017multipliers} gives a precise description of $\mon\mathcal P(^m\ell_{\infty})$ as $\ell_{\frac{m-1}{2m}, \infty}$ and lower and upper inclusions for $\mon H_{\infty} (B_{\ell_{\infty}})$ that, although not optimal, are pretty tight. The study for $1 < r <\infty$ was started in \cite{defant2009domains}  and continued in \cite{bayart2016monomial}, where several interesting results in this direction for polynomials and bounded holomorphic functions were obtained. To our best knowledge, nothing has been done so far to describe the set of monomial convergence of the holomorphic functions of bounded type. In this note we make progress towards the description of these set of monomial convergence  in the case $1 < r \leq 2$.\\

In Theorem~\ref{thm: main theorem hb} we provide a complete characterization of the set of monomial convergence of the space of holomorphic functions of bounded type for $1 < r \le 2$ as
\[
\mon H_b(\ell_r) = \left\lbrace z \in \CC^\NN :  \sup_{n \ge 1} \frac{\sum_{l=1}^n z_l^* }{\log(n+1)^{1 - \frac{1}{r}}} < \infty \right\rbrace.
\]
The proof is given in Section~\ref{seccion Hb} and the main tool developed is a decomposition of the multi-indices (in an even and a pure tetrahedral part), which allows us to split the monomial expansion in different pieces, for which we are able to find proper bounds.\\

Regarding set of monomial convergence of bounded holomorphic functions on $B_{\ell_{r}}$ is considered, there are a number of deep results (see \cite[Example~4.9 (1)(a)]{defant2009domains}) that in the case we are dealing with here ($1 < r \leq 2$) imply 
\begin{equation} \label{saintsaens}
B_{\ell_{r}} \cap \ell_{1} \subsetneq \mon H_{\infty} (B_{\ell_{r}}) \subseteq B_{\ell_{r}} \cap \ell_{1+\varepsilon}  \text{ for every  } \varepsilon >0. 
\end{equation}
We give here some upper and lower inclusions, in the spirit of the ones obtained for $H_{\infty} (B_{\ell_{\infty}})$. We show in Theorem~\ref{teo: mon H_infty} that
\begin{multline*}
\Big\{ z \in \mathbb{C}^{\mathbb{N}} \colon  2e C_r \left(  \limsup_{n \to \infty} \frac{\sum_{k=1}^n z_k^*}{\log(n+1)^{1-1/r}}\right)^r  + \| z \|_{\ell_r}^r < 1  \Big\} \subset 
\mon H_{\infty} (B_{\ell_{r}}) \\
\subset \Big\{ z \in B_{\ell_r} \colon \limsup_{n \to \infty} \frac{\sum_{k=1}^n z_k^*}{\log(n+1)^{1-1/r}} \leq 1 \Big\} \,,
\end{multline*}
where $0< C_r \le \left( \dis\sum_{k = 1}^\infty \frac{\log(k+1)^{r-1}}{k^r} \right)^{1/r}$ and depends on the interplay between $\ell_r$ and the Marcinkiewicz sequence space $m_{\Psi_r}$ (see Remark~\ref{rem: desig coord}).
Let us point out that this is connected with the question stated in \cite[Remark~5.8]{bayart2016monomial}.
We will see in Remark~\ref{georgewinston} that these lower and upper inclusions recover \eqref{saintsaens}.\\

Regarding $m$-homogeneous polynomials we know from \cite[Theorem~5.1]{bayart2016monomial} and \cite[Example~4.6]{defant2009domains} that $\ell_{q- \varepsilon} \subset \mon \mathcal P(^m\ell_r) 
\subset \ell_{q,\infty}$ for every $\varepsilon>0$ (where $1 < r \leq 2$ and $q:= (mr')'$). Using elementary methods we show in Theorem~\ref{teo ellq} that we can even take $\varepsilon =0$ (this  proves a conjecture made by Defant, Maestre and Prengel in \cite{defant2009domains}). We go one step further, showing in Theorem~\ref{teo: main poli} that 
\[
\ell_{q, \frac{m}{\log m}} \subset \mon \mathcal{P}(^m\ell_r)
\]
for every $m \geq 5$ (we also give lower inclusions for $m\le 4$). The proof is technically involved and uses interpolation of linear operators defined  on cones. All this is presented in Section~\ref{mhomogeneous}.\\

Finally, as a byproduct, in Section~\ref{aplicaciones} we provide correct estimates of the asymptotic growth of the mixed-$(p,q)$ unconditional constant (a notion by Defant, Maestre and Prengel in \cite[Section~5]{defant2009domains}) as $n$ tends to infinity for every $1 \leq p, q \leq \infty$; closing the gap of the remaining cases of  \cite{galicer2016sup}.

\section{Preliminaries}

\label{preliminares}

For every $x,y \in \CC^\NN$ we denote by $|x|$  the sequence $(|x_1|,|x_2|, \ldots,|x_n|, \ldots )$. If $|x_i| \le |y_i| $ for every $i \in \NN$ we write $\vert x \vert \leq \vert y \vert$. A Banach sequence space is a Banach space $(X, \| \cdot \|_X)$  such that  $ \ell_1 \subset X \subset \ell_\infty$ 
satisfying that, if $ x \in \mathbb{C}^{\mathbb{N}}$ and $y \in X$ with  $|x| \le |y|$, then $x \in X$ and $\| x \|_X \le \| y \|_X$. 
A non-empty open set $\Rr \subset X$  is called a Reinhardt domain if given $ x \in \CC^\NN$ and $y \in \Rr$ such that $ |x| \le |y| $ then $x \in \Rr $.
Given a bounded sequence $x$ its decreasing rearrangement $x^*$ is the sequence defined as $x_n^* = \inf \lbrace \sup_{j \in \NN \setminus J}|x_j| : J \subset \NN , \card(J) < n \rbrace$. A Banach sequence space $(X, \| \cdot \|_X)$ is said to be symmetric if $x^* \in X$ whenever $x \in X$  and, moreover $\| x \|_X = \| x^* \|_X$. A set $A \subset X$ is symmetric if $x \in A $ if and only if $x^* \in A$.
For every $x \in c_0$ there is some injective mapping $\sigma : \NN \to \NN$ such that $x_n^* = |x_{\sigma(n)}|$ for all $n \in \NN$. We will say that a sequence $x \in \mathbb{C}^{\mathbb{N}}$ is decreasing whenever $\vert x \vert$ is decreasing.

We are going to deal basically with three classes of Banach sequence spaces: the classical Minkowski $\ell_{r}$ spaces, the Lorentz $\ell_{p,q}$ spaces and the Marcinkiewicz sequence spaces. Let us recall some 
definitions.  For $1 \leq p, q \leq \infty$ the space $\ell_{p,q}$ consists of those sequences $z$ for which (we use the convention $\frac{1}{\infty}=0$)
\[
\Vert z \Vert_{\ell_{p,q}} := \Big\Vert \big( z^{*}_{n} n^{\frac{1}{p} - \frac{1}{q}} \big)_{n=1}^{\infty}  \Big\Vert_{\ell_{q}} < \infty \,.
\]
Observe that in general this is a quasi-norm and only defines a norm for $1 \le q \le p \le \infty$. For $z \in \ell_{p,q}$ we define 
$$
\Vert z \Vert_{\ell_{p,q}}^* := \left( \sum_{n=1}^\infty n^{\frac{q}{p} -1} \left( \frac{1}{n} \sum_{k=1}^n z_k^* \right)^q  \right)^{1/q}.
$$
It should be noted (see \cite[Lemma~4.5]{bennett1988interpolation}) that for $1 \le p,q \le \infty$ and $z \in \ell_{p,q}$, it holds
\begin{equation*} 
\Vert z \Vert_{\ell_{p,q}} \le \Vert z \Vert_{\ell_{p,q}}^* \le p' \Vert z \Vert_{\ell_{p,q}},
\end{equation*}
so we can always work with the quasi-norm $\Vert \cdot \Vert_{\ell_{p,q}}$ and treat $(\ell_{p,q}, \Vert \cdot \Vert_{\ell_{p,q}})$ as a Banach sequence space at the expense of $p'$ (the conjugate exponent of $p$) as a price every time we do so .
Let  $\Psi = (\Psi (n) )_{n = 0}^\infty$ be an increasing sequence of nonnegative real numbers with $\Psi (0) = 0$ and $\Psi(n) > 0$ for every $n \in \NN$. These functions are usually known as symbols. The Marcinkiewicz sequence space associated to the symbol $\Psi$, denoted by $m_\Psi$, is the vector space of all bounded sequences $(z_n)_n$ such that
\[ 
\| z \|_{m_\Psi}:= \sup_{n \ge 1} \frac{\sum_{k=1}^n z_k^* }{\Psi(n)} < \infty.
\]

An $m$-homogeneous polynomial in $n$ variables is a function $P$ of the form
\[
P(z) = \sum_{ \genfrac{}{}{0pt}{}{\alpha \in \NN_{0}^{n}}{\alpha_{1} + \cdots + \alpha_{n}=m} } c_{\alpha} z_{1}^{\alpha_{1}} \cdots z_{n}^{\alpha_{n}}.
\]

Given $\alpha \in \mathbb{N}_{0}^{n}$ we write $\vert \alpha \vert = \alpha_{1} + \cdots + \alpha_{n}$ and $\Lambda (m,n) = \{ \alpha \in \mathbb{N}_{0}^{n} \colon \vert \alpha \vert =m  \}$. We also consider the set $\mathcal{J}(m,n) = \{ \mathbf{j} = (j_{1} , \ldots , j_{m}) \in \mathbb{N}^{m} \colon  1 \leq j_{1} \leq \cdots \leq j_{m} \leq n \}$. Each $\alpha \in \Lambda (m,n)$ defines $\mathbf{j}_{\alpha} = (1, \stackrel{\alpha_{1}}{\ldots} 1, 2, \stackrel{\alpha_{2}}{\ldots} 2, \ldots, n, \stackrel{\alpha_{n}}{\ldots} n)  \in \mathcal{J}(m,n)$.
Conversely, each $\mathbf{j} \in \mathcal{J}(m,n)$ defines $\alpha \in \Lambda (m,n)$ by $\alpha_{k} = \card \{ i \colon j_{i} = k \}$. In this way these two indexing sets are injective and, denoting $z_{1}^{\alpha_{1}} \cdots z_{n}^{\alpha_{n}} = z^{\alpha}$ and $z_{j_{1}} \cdots z_{j_{m}} = z_{\mathbf{j}}$ we can write each homogeneous polynomial in two alternative ways
\begin{equation} \label{eternidade}
P(z) = \sum_{\alpha \in \Lambda(m,n)} c_{\alpha} z^{\alpha}
= \sum_{1 \leq j_{1} \leq \cdots \leq j_{m} \leq n} c_{j_{1}, \ldots , j_{m}} z_{j_{1}} \cdots z_{j_{m}} 
= \sum_{\mathbf{j} \in  \mathcal{J}(m,n)} c_{\mathbf{j}}z_{\mathbf{j}} .
\end{equation}
We will freely change from the $\alpha$ to the $\mathbf{j}$ notation whenever it is more convenient (always assuming that $\alpha$ and $\mathbf{j}$ are related to each other). 
We write
\[
\vert \mathbf{j} \vert = \card \{ \mathbf{i} \in \mathbb{N}^{m} \colon \text{ there exists a permutation } \sigma \text{ of } {1, \ldots , m} \text{ so that } i_{\sigma(k)} = j_{k} \text{ for all } k  \} \,.
\]
Note that if $\mathbf{j}$ and $\alpha$ are associated to each other, then
\begin{equation}\label{strauss}
\vert \mathbf{j} \vert = \frac{m!}{\alpha_{1}! \cdots \alpha_{n}! } = \frac{m!}{\alpha!}.
\end{equation}
We will sometimes denote this by $\vert [\alpha] \vert$.
We write $\mathcal{P}(^{m}  \mathbb{C}^{n} )$ for the space of all $m$-homogeneous polynomials in $n$ variables. Each $\ell_{r}$-norm on $\mathbb{C}^{n}$ induces a different (though all equivalent) norm $\Vert P \Vert_{\mathcal{P}(^{m}  \ell_{r}^{n} )} = \sup_{\Vert z \Vert_{r} \leq 1} \vert P(z) \vert$.\\

We follow the theory of holomorphic functions on arbitrary Banach spaces as presented in \cite{mujica1986, dineen1999complex}. 
If $X$ is a (finite or infinite dimensional) Banach space, a function $P : X \to \mathbb{C}$ is a (continuous) $m$-homogeneous polynomial if there exists a (unique) continuous symmetric $m$-linear form (denoted by $\check{P}$) on $X$ such that $P(x) = \check{P} (x, \ldots ,x)$ for every $x$.
A function $f : U \to \mathbb{C}$ (where $U$ is some open subset of a Banach space $X$) is holomorphic if it is Fr\'echet differentiable at every point of $U$. If $U$ is balanced there are $P_{m}(f)$ for $m=0,1,2,\ldots$, each an $m$-homogeneous polynomial on $X$, such that $f = \sum_{m} P_{m}(f) $ uniformly on $U$. The space of all holomorphic functions on $U$ is denoted by $H(U)$. The space of bounded holomorphic functions on $B_{X}$ (the open unit ball of $X$) with the norm $\Vert f \Vert = \sup_{\Vert x \Vert \leq 1} \vert f(x) \vert$ is denoted by $H_{\infty}(B_{X})$. The space of $m$-homogeneous polynomials on $X$ is denoted by $\mathcal{P} (^{m} X)$, and is endowed with the norm $\Vert P\Vert = \sup_{\Vert x \Vert \leq 1} \vert P(x) \vert$. Every homogeneous polynomial is entire (holomorphic on $X$) and, then, its coefficients can be computed through \eqref{dowland}. Let us note that $c_{\alpha} (P) \neq 0$ only if $\vert \alpha \vert =m$ and that, if $\mathbf{j} \in \mathcal{J}(m,n)$ is associated to $\alpha$, then
\[
c_{\alpha}(P) = \frac{m!}{\alpha !} \check{P} (e_{j_{1}}, \ldots , e_{j_{m}}) \,.
\]
An entire function is said to be of bounded type if it is bounded on every bounded set of $X$. The space of entire functions of bounded type is denoted by $H_{b}(X)$. It is a Fr\'echet space with the family of seminorms defined by $p_{n} (f) =  \sup_{\Vert x \Vert \leq n} \vert f(x) \vert$.\\

We denote by $\mathbb{N}_{0}^{(\mathbb{N})}$ the set of eventually zero multi-indices. In other words, $\mathbb{N}_{0}^{(\mathbb{N})} = \bigcup_{n=1}^{\infty} \mathbb{N}_{0}^{n} \times \{0\}$. From now on we will identify $\mathbb{N}_{0}^{n} \times \{0\}$ with $\mathbb{N}_{0}^{n}$ without further notice.

\section{Rearrangement families of holomorphic functions.} \label{section: r.f.}

A very useful tool in the study of sets monomial convergence (see \cite{bayart2017multipliers}) is that usually, a sequence belongs to the set of monomial convergence if and only if its decreasing rearrangement does (see also \cite{defant2008bohr}). We isolate this property, and say in this case that  $\Ff \subset H(\Rr)$ is a \emph{rearrangement family} (where $\Rr$ is a  Reinhardt domain in a Banach sequence space $X$). In \cite{bayart2017multipliers} it was proved that $H_{\infty} (B_{c_{0}})$ and $\mathcal{P}(^{m} c_{0})$ are rearrangement families. The fact that this is also the case for $\ell_{r}$ for $1 \leq r < \infty$ is implicitly  used in \cite{bayart2016monomial}. Our aim now is to find other rearrangement families of holomorphic functions (compare this with \cite[Chapter~7]{schluters2015unconditionality} where similar results appear).\\

To this purpose we introduce another concept. We say a family $\Ff \subset H(\Rr)$ is \emph{linearly balanced} if $f \circ T\restrict{\Rr} \in \Ff$ for every $f \in \Ff$ and $T: X \to X$ linear with $ \| T \| =1$ and $T(\mathcal{R}) \subset \mathcal{R}$.

\begin{remark} \label{hindemith}
Rather straightforward arguments show that $H_b(X)$,  $\Aa_u(B_X)$ (all uniformly continuous and holomorphic functions on $B_{X}$), $H_\infty(B_X)$ and $\Pp(^m X)$ for every $m \ge 2$ are linearly balanced families.
\end{remark}

\begin{theorem}\label{teo: balanc + C0 es reord}
Let $\Rr$ be a symmetric Reinhardt domain of a symmetric Banach sequence space $X$ and $\Ff \subset H(\Rr)$ a linearly balanced family such that $\mon \Ff \subset c_0$, then $\Ff$ is a rearrangement family.
\end{theorem}

We give a series of preliminary results needed for the proof of Theorem~\ref{teo: balanc + C0 es reord}. Given an injective mapping $\sigma : \NN \to \NN$ we define two mappings in the following way. First
\begin{equation}\label{eq: op T}
\begin{split}
T_\sigma : \CC^{\NN} & \to \CC^{\NN} \\
x & \mapsto (x_{\sigma(k)})_{k \in \NN} \,.
\end{split}
\end{equation}
Second, $S_\sigma : \CC^\NN \to \CC^\NN$ is defined for $x \in \CC^{\NN}$ by 
\begin{equation}\label{eq: op S}
(S_{\sigma}x)_k =\begin{cases}
0 & \text{ if } k \notin \sigma(\NN) \\
x_{\sigma^{-1}(k)} & \text{ if } k \in \sigma(\NN).
\end{cases}    
\end{equation}
Both are clearly  linear and $T_{\sigma} (S_{\sigma} x) = x$ for every $x$.

\begin{remark}\label{rem: T,S achican la estrella }
Let us see now how these two mappings behave with the decreasing rearrangement of a bounded sequence $x$. Fixed $n \in \NN$ and $J \subset \NN$ such that $\card(J) < n$ we have
\[
    \sup_{\sigma(j) \in \NN \setminus J} |x_{\sigma(j)}|  
     =  \sup_{ j \in (\NN \setminus J) \cap \sigma(\NN)} |x_j| 
     \le \sup_{ j \in \NN \setminus J} |x_j| \,.
\]
Thus
\[
    \big( T_\sigma(x) \big)^*_n 
    =  \inf \lbrace \sup_{\sigma(j) \in \NN \setminus J} |x_{\sigma(j)}| : J \subset \NN, \text{card}(J) < n \rbrace 
    \le  \inf \lbrace \sup_{j \in \NN \setminus J} |x_j| : J \subset \NN, \text{card}(J) < n \rbrace = x_n^* .
\]
That is, $T_\sigma(x)^* \le x^* $. A similar argument shows that $(S_\sigma x)^* = x^*$. 
\end{remark}

The following lemma shows that the restrictions of $S_\sigma$ and $T_\sigma$ to symmetric Banach sequence spaces are endomorphisms of norm $1$.

\begin{lemma}\label{lem: b.d. T,S}
Let $X$ be a symmetric Banach sequence space and $\sigma : \NN \to \NN$ an injective mapping. Then $T_\sigma, S_\sigma: X \to X$ defined by \eqref{eq: op T} and \eqref{eq: op S} respectively are well defined, $\| T_\sigma \| =1$ and $ S_\sigma$ is an isometry. 
\end{lemma}

\begin{proof}
Remark~\ref{rem: T,S achican la estrella } together with the symmetry of the space imply that both operators are well defined, that $S_{\sigma}$ is an isometry and $\Vert T_{\sigma} \Vert \leq 1$. 
The fact that $\| T_\sigma \| =1$ follows from the equality $T_\sigma (S_\sigma x_{0})=x_0$.
\end{proof}

Now we are able to give the proof of Theorem~\ref{teo: balanc + C0 es reord}.

\begin{proof}[Proof of Theorem~\ref{teo: balanc + C0 es reord}]
To begin with we take $z \in \mon \mathcal{F}$ and see that $z^{*} \in \mon \mathcal{F}$. As $\mon \Ff \subset c_0$ there is some injective mapping $\sigma: \NN \to \NN $ such that $z_k^* = |z_{\sigma(k)}|$ for every $k \in \NN$. Observe that $|T_\sigma(z)| = z^*$. We take $f \in \mathcal{F}$, then $f \circ T_{\sigma}$ also belongs to $\mathcal{F}$ and what we want to see first is that, if $\alpha(\sigma) \in \NN_0^{(\NN)}$ denotes the multi-index that fulfils $T_\sigma(z)^\alpha = z^{\alpha(\sigma)}$, then
\begin{equation} \label{carlomagno}
c_{\alpha} (f) = c_{\alpha (\sigma)} (f \circ T_{\sigma})
\end{equation}
for every $\alpha$. Take, then, some $\alpha \in \mathbb{N}_{0}^{(\mathbb{N})}$ and set $N= \max\{k \colon \alpha_{k} \neq 0 \}$. On one hand we have
\[
(f \circ T_{\sigma}) (w) = \sum_{\beta \in \mathbb{N}_{0}^{N}} c_{\beta} (f \circ T_{\sigma}) w^{\beta} \,,
\]
for all $w \in \mathbb{C}^{N} \cap \mathcal{R}$. Define $M= \max \{ \sigma(k) \colon k=1, \ldots, N  \}$ and note that $T_{\sigma} (w) \in  \mathbb{C}^{M} \cap \mathcal{R}$. Thus
\[
 (f \circ T_\sigma)(w) = f(T_\sigma(w))= \sum_{\gamma \in \NN_0^{N} } c_{\gamma} (f) T_\sigma(w)^{\gamma} =  \sum_{\gamma \in\NN_0^{N} } c_{\gamma} (f) w^{\gamma(\sigma)}.
\]
The uniqueness of the Taylor coefficients gives \eqref{carlomagno}. Once we have this we obtain (recall that $f \circ T_{\sigma} \in \mathcal{F}$ and $z \in \mon \mathcal{F}$)
\[
\sum_{\alpha \in \NN_0^{(\NN)}} |c_\alpha(f)(z^*)^\alpha| 
    =   \sum_{\alpha \in \NN_0^{(\NN)}} |c_\alpha(f)| |(T_\sigma(z))^\alpha|
= \sum_{\alpha \in \NN_0^{(\NN)}} |c_{\alpha(\sigma)}(f \circ T_\sigma)| |z^{\alpha(\sigma)}| \\
    \le  \sum_{\alpha \in \NN_0^{(\NN)}} |c_\alpha(f \circ T_\sigma) z^\alpha| < \infty,
\]
which proves our claim.\\

For the converse, suppose $z^* \in \mon \Ff$. Again, as $\mon \Ff \subset c_0$, there is some injective mapping $\sigma: \NN \to \NN $ such that $z_k^* = |z_{\sigma(k)}|$ for every $k \in \NN$. Now it will be useful to notice $| z | = S_\sigma(z^*)$. Given $f \in \Ff $ we have
\begin{align}\label{eq: taylor reord z*}
    \sum_{\alpha \in \NN_0^{(\NN)}} |c_\alpha(f) z^\alpha| 
    = &  \sum_{\alpha \in \NN_0^{(\NN)}} |c_\alpha(f)| |(S_\sigma(z^*))^\alpha|.
\end{align}
Besides,
\[
\sum_{\alpha \in \NN_0^N } c_\alpha (f \circ S_\sigma) w^\alpha = f (S_\sigma(w)) = \sum_{\alpha \in \NN_0^N } c_\alpha (f) S_\sigma(w)^\alpha = \sum_{\alpha \in \NN_0^N } c_\alpha (f) S_\sigma(w)^\alpha .
\]
Observe that for $\alpha \in \NN^{(\NN)}$, if there is $k \in \NN \setminus \sigma(\NN)$ such that $\alpha_ k \neq 0$ then $S_\sigma(w)^\alpha = 0$, otherwise we define $\alpha(\sigma^{-1}) \in \NN^{(\NN)}$ as the only multi-index which fulfils $S_\sigma(w)^\alpha = w^{\alpha(\sigma^{-1})}$. By the uniqueness of the coefficients of the Taylor expansion for $ f  \circ S_\sigma : \CC^N \to \CC$ it follows
\begin{equation*}
c_\alpha(f) S_\sigma(z^*)^\alpha = \begin{cases}
0 & \text{ if there is } k \notin \sigma(\NN) \text{ such that } \alpha_k \neq 0 \\
c_{\alpha(\sigma^{-1})}(f \circ S_\sigma) (z^*)^{\alpha(\sigma^{-1})}  & \text{ otherwise,}
\end{cases}    
\end{equation*}
then
\begin{equation}\label{basta la desi para z*}
\begin{split}
    \sum_{\alpha \in \NN_0^{(\NN)}}|c_\alpha(f) z^\alpha| 
  =   \sum_{\alpha \in \NN_0^{(\NN)}} &   |c_\alpha(f) | |(S_\sigma(z^*))^\alpha| \\
&    =  \sum_{\alpha \in (\sigma(\NN) \cup \lbrace 0 \rbrace)^{(\NN)}} 
\big\vert  c_{\alpha(\sigma^{-1})}(f \circ S_\sigma) (z^*)^{\alpha(\sigma^{-1})} \big\vert 
    \le  \sum_{\alpha \in \NN_0^{(\NN)}} |c_\alpha(f \circ S_\sigma)| |(z^*)^\alpha| < \infty,
\end{split}
\end{equation}
as we wanted.
\end{proof}

\begin{remark}\label{rem: familias con mon en c0}
Let $\Rr$ be a symmetric Reinhardt domain in a Banach sequence space $X$ and consider  a family of homolorphic functions $\Ff \subset H(\Rr)$ such that for some $m \ge 2$ the space $\Pp(^m X)$ lies inside $\Ff$. Then, as $X \subset \ell_\infty$ continuously  we have $\Pp(^m \ell_\infty) \subset \Pp(^m X) \subset \Ff$. With this, \cite[Theorem~2.1]{bayart2017multipliers} yields
\[
\mon \Ff \subset \mon \Pp(^m \ell_\infty) = \ell_{\frac{2m}{m-1},\infty}\subset c_0.
\]
\end{remark}

\begin{corollary}\label{cor: familias de reordenamiento}
For every symmetric Banach sequence space $X$ the families of holomorphic functions $H_b(X), \Aa_u(B_X), H_\infty(B_X)$ and $\Pp(^m X)$ with $m \ge 2$ are rearrangement families.
\end{corollary}

\begin{proof}
Each of these families satisfies the condition in Remark~\ref{rem: familias con mon en c0}. Then Remark~\ref{hindemith} and Theorem~\ref{teo: balanc + C0 es reord} give the conclusion.
\end{proof}

\begin{remark}
As we have already pointed out, we are mainly interested in $H_{\infty} (B_{\ell_{r}})$,  $H_{b}(\ell_{r})$ and $\mathcal{P}(^{m} \ell_{r})$. The set of monomial convergence of each one of these spaces is, by Remark~\ref{rem: familias con mon en c0}  contained in $c_{0}$. But, as matter of fact, we can say more. By \cite[Proposition~20.3]{defant2019libro} we have 
$\mon H_{\infty} (B_{\ell_{r}}) \subseteq B_{\ell_{r}}$. Noting that every functional $f \in \ell_{r}^{*}$ belongs to $H_{b} (\ell_{r})$ and using the definition of the set of monomial convergence we have $\mon H_{b}(\ell_{r}) \subseteq \ell_{r}$. Finally, exactly the same argument as in \cite[Remark~10.7]{defant2019libro} shows that $\mon \mathcal{P}(^{m} \ell_{r}) \subseteq \mathcal{P}(^{1} \ell_{r}) = \mon \ell_{r}^{*}=\ell_{r}$.
\end{remark}

\section{Monomial convergence for holomorphic functions of bounded type on \texorpdfstring{$\ell_r$}.}\label{seccion Hb}

We can now describe the set of monomial convergence of $H_{b}(\ell_{r})$ for $1 < r \leq 2$. It happens to be a Marcinkiewicz space $m_{\Psi_r}$ where the symbol is given by 
\[
\Psi_r(n) := \log(n + 1)^{1 - \frac{1}{r}},
\]
for $n \in \NN_0$.

\begin{theorem}\label{thm: main theorem hb}
For $1 < r \le 2$,  
\[
\mon H_b(\ell_r) = m_{\Psi_r}:=  \left\lbrace z \in \CC^\NN :  \sup_{n \ge 1} \frac{\sum_{k=1}^n z_k^* }{\log(n+1)^{1 - \frac{1}{r}}} < \infty \right\rbrace.
\]
\end{theorem}

We handle the upper and the lower inclusions separately in the following two sections.

\subsection{The upper inclusion \texorpdfstring{$\boldsymbol{\mon H_b(\ell_r) \subset m_{\Psi_r} }$}.}

Typically, the way to prove upper inclusions for a set of monomial convergence goes through providing polynomials satisfying certain convenient properties. Over the last years probabilistic techniques have shown to be extremely helpful to find such polynomials. This is, for instance, what is done in \cite[Theorem~2.2]{bayart2017multipliers}, where the probabilistic device is the well known Kahane-Salem-Zygmund inequality. Here we follow essentially the same lines, replacing the polynomials provided by this inequality by other ones. Following techniques of Boas and Bayart (see \cite{boas2000majorant},  \cite{bayart2012maximum} and also \cite[Corollary~17.6]{defant2019libro}) for every $1 \leq r  \leq 2$ there is a constant $C_{r}>0$ such that for all $n$ and $m\ge 2$ we can find a choice of signs $(\varepsilon_{\alpha})_{\alpha}$ so that
\begin{equation}\label{polinomios de Bayart}
\sup_{\Vert z \Vert_{r} < 1} \Big\vert \sum_{\alpha \in \Lambda(m,n)} \varepsilon_{\alpha} \frac{m!}{\alpha!} z^{\alpha} \Big\vert \leq C_{r}
  (\log(m) m! )^{1- \frac{1}{r}} \;n^{1- \frac{1}{r}}\,.
\end{equation}
These polynomials are the main tool for the proof of the upper inclusion. We also need the following result, an extension of \cite[Lemma~4.1]{defant2009domains} whose proof follows the same lines. 

\begin{lemma}\label{relacion mon}
Let  $\mathcal{R}$ be a Reinhardt domain in a Banach sequence space $X$ and let $(\mathcal F,(q_{n})_{n})$ be a Fr\'echet space of holomorphic functions  continuously included in $H_b(\mathcal{R})$.  Then, for each $z\in \mon(\mathcal F)$, there exist $C>0$ and $n$ such that 
\[
\sum_{\alpha\in \mathbb N_0^{(\mathbb N)}}  |c_{\alpha} z^\alpha| \le C q_{n}(f).
\]
for every $f\in \mathcal F$.
In particular, if $z \in \mon H_b(X)$, there exists $C>0$,  such that 
\[
\sum_{\alpha\in \Lambda(m,n)} |c_{\alpha}(P) z^\alpha| \le C^m \|P\|_{\mathcal P(^mX)},
\]
for every $P\in \mathcal P(^mX)$.
\end{lemma}

We have now everything at hand to proceed with the proof of the upper inclusion.

\begin{proof}[Proof of the upper inclusion in Theorem~\ref{thm: main theorem hb}]
Fix $1 < r \le 2$ and choose $z \in \mon H_{b} (\ell_{r})$. Now fix $n,m$, choose signs as in \eqref{polinomios de Bayart} and define the polynomial $P(w):=  \sum_{\alpha \in \Lambda(m,n)} \varepsilon_{\alpha} \frac{m!}{\alpha!} w^{\alpha}$. By Corollary~\ref{cor: familias de reordenamiento} we know that $z^{*} \in \mon H_b(\ell_r)$. Using first the multinomial formula, then Lemma~\ref{relacion mon} and finally \eqref{polinomios de Bayart} we have
\begin{equation} \label{elenitagentil}
\begin{split}
\left( \sum_{j = 1}^n |z_j^*|  \right)^m = \sum_{ \alpha \in \Lambda(m,n)} \frac{m!}{\alpha!} & |(z^*)^\alpha| 	 =  \sum_{ \alpha \in \Lambda(m,n)} \Big\vert \varepsilon_\alpha \frac{m!}{\alpha!} (z^*)^\alpha \Big\vert \\
& \leq C_{z^{*}}^m \sup_{u \in B_{\ell_{r}^{n}}} \left| \sum_{ \alpha \in \Lambda(m,n)} \varepsilon_\alpha \frac{m!}{\alpha!} u^\alpha\right|_{\Pp(^{m} \ell_r^n)} 
\leq C_{z^{*},r}^m (\log(m) m! n)^{1 - \frac{1}{r}}.
\end{split}
\end{equation}
Taking the power $1/m$ and using Stirling's formula ($ m! \le \sqrt{2 \pi m} e^{\frac{1}{12m}} m^m e^{-m}$) yield
\begin{equation} \label{desigBayart}
 \sum_{j = 1}^n |z_j^*| \le C_{z^*,r} \left[\log(m)^{\frac{1}{m}} (2 \pi m)^{\frac{1}{2m}} e^{\frac{1}{12m^2}}  \frac{m}{e} n^{\frac{1}{m}}\right]^{1 - \frac{1}{r}}.   
\end{equation}
Finally, choosing $m = \lfloor \log(n+1) \rfloor$ gives that the term $\frac{1}{\log(n+1)^{1-\frac{1}{r}}} \sum_{k = 1}^n |z_n^*|$ (for every $n \ge 2$) is bounded independently of $n$, so $z \in m_{\Psi_r}$.
\end{proof}

\subsection{The lower inclusion \texorpdfstring{$\boldsymbol{m_{\Psi_r} \subset \mon H_b(\ell_r)}$}.}

We face now the proof of the lower inclusion in  Theorem~\ref{thm: main theorem hb}. The main tool is the following result, the proof of which requires some work, that we perform all along this section.

\begin{theorem}\label{teo: hypercontrac 1 < r <2}
Fix $1 < r \le 2$. For every $\varepsilon > 0$ there is $C_{r}=C_r(\varepsilon)>0$ such that for every $m,n \in \NN$, every $m$-homogeneous polynomial in $n$ complex variables  $P$ and every $z \in \mathbb{C}^{n}$, we have 
\[
 \sum_{\jj \in \Jj(m,n)} | c_\jj(P) z_\jj^{*} |  
\le C_r(\varepsilon) m^{2+\frac{1}{r}} ((1+\varepsilon)2e)^{\frac{m}{r}}\| \id: m_{\Psi_r} \to \ell_r \|^{m} 
\| z \|_{m_{\Psi_r}}^m \| P \|_{\Pp(^m \ell_r^n)}\,.
\]
\end{theorem}

Before we start with the proof of this result, let us see how, having it at hand, we can prove the lower inclusion we are aiming at.

\begin{proof}[Proof of the lower inclusion in Theorem~\ref{thm: main theorem hb}]
Choose $z \in m_{\Psi_r}$ and let us see that $z \in \mon H_{b}(\ell_{r})$. By Corollary~\ref{cor: familias de reordenamiento} we may assume without loss of generality $z=z^{*}$. Given $f \in H_b(\ell_r)$ (recall that we denote $P_m(f)$ for the $m$-homogeneous part of the Taylor expansion) and Theorem~\ref{teo: hypercontrac 1 < r <2} (with $\varepsilon = 1$) gives
\begin{multline*}
\sum_{\alpha \in \mathbb{N}_0^{(\mathbb{N})} } \vert  c_\alpha(f) z^\alpha \vert = \sup_{n \in \mathbb{N}} \sum_{m = 0}^\infty \sum_{\jj \in \mathcal J(m,n) } \vert c_{\jj}(f) z_{\jj} \vert \\
  \leq \sup_{n \in \mathbb{N}} \sum_{m = 0}^\infty  C_r m^{2+\frac{1}{r}} (4e)^{\frac{m}{r}} \| \id \|^m  \| z \|_{m_{\Psi_r}}^m 
  \sup_{u \in B_{\ell_{r}^{n}}} \left\vert \sum_{\jj \in \mathcal J(m,n) } c_{\jj}(f) u_{\jj} \right\vert \\
 = C_r \sum_{m = 0}^\infty (m^{(2+\frac{1}{r})\frac{1}{m}} (4e)^{\frac{1}{r}} \| \id \| \| z \|_{m_{\Psi_r}})^m \| P_m(f) \|_{\Pp(^m \ell_r)}.
\end{multline*}
Let us see that this sum is finite. Take $R>\sup_{m} \big( m^{(2+\frac{1}{r})\frac{1}{m}} (4e)^{\frac{1}{r}} \| \id \| \| z \|_{m_{\Psi_r}} \big)$, then by the homogeneity of $P_m(f)$
\begin{multline*}  
\sum_{m = 0}^\infty (m^{(2+\frac{1}{r})\frac{1}{m}} (4e)^{\frac{1}{r}} \| \id \| \| z \|_{m_{\Psi_r}})^m \| P_m(f) \|_{\Pp(^m \ell_r)} \\ 
= \sum_{m = 0}^\infty \left(\frac{m^{(2+\frac{1}{r})\frac{1}{m}} (4e)^{\frac{1}{r}} \| \id \| \| z \|_{m_{\Psi_r}}}{R}\right)^m \sup_{w \in R\cdot B_{\ell_r}} | P_m(f)(w)| \\
 \leq  \sum_{m = 0}^\infty \left(\frac{m^{(2+\frac{1}{r})\frac{1}{m}} (4e)^{\frac{1}{r} } \| \id \| \| z \|_{m_{\Psi_r}}}{R}\right)^m  \sup_{w \in R\cdot B_{\ell_r}} | f(w)|< \infty,
\end{multline*}
where the last step is due to Cauchy's inequality. This completes the proof.
\end{proof}

We start now the way to the proof of Theorem~\ref{teo: hypercontrac 1 < r <2}. We begin with a simple remark.

\begin{remark}\label{rem: desig coord}
If $z \in m_{\Psi_r}$, then
\[
n \vert z_n^{*} \vert \le \sum_{l = 1}^n z_l^* \le \| z \|_{m_{\Psi_r}} \log(n +1)^{\frac{1}{r'}}.
\]
That is
\[
\vert z_n^{*} \vert  \le \| z \|_{m_{\Psi_r}} \frac{\log(n+1)^\frac{1}{r'}}{n}
\]
for every $n \in \NN$. This gives
\[
\sum_{j=1}^{n} \vert z_{j} \vert^{r} \leq \sum_{j=1}^{n} \vert z_{j}^{*} \vert^{r} 
\leq \| z \|_{m_{\Psi_r}}^{r} \sum_{j=1}^{n}  \frac{\log(j+1)^\frac{r}{r'}}{j^{r}} \,.
\]
This implies $\Vert \id : m_{\Psi_r} \to \ell_{r} \Vert \leq \Big( \sum_{j=1}^{\infty}  \frac{\log(j+1)^\frac{r}{r'}}{j^{r}} \Big)^{1/r}$ (note that this series is convergent for $1 < r$).
\end{remark}

Our first ingredient is the following lemma, that follows with a careful analysis of the proof of \cite[Lemma~3.5]{bayart2016monomial}, that relates the summability of certain coefficients of a polynomial and its uniform norm in $\ell_r^n$. It has been very useful to provide a proof `at an elementary level' (in the sense that it does not require tools from the local theory of Banach space)  of the asymptotic growth of the unconditional constant of the space of $m$-homogeneous polynomials on $\ell_r^n$   as $n$ goes to infinite  with suitable care on the dependence of $m$ (in fact this has been proved for general index sets, see \cite[Theorem~3.2]{bayart2016monomial}). As a consequence the behaviour of the Bohr radii of holomorphic functions on $\ell_r$  for $1 \leq r \leq 2$ has been described in \cite[Theorem~3.9]{bayart2016monomial}. It has recently been used also to  study the asymptotic growth of the mixed Bohr radii in \cite{galicer2017mixed}. In some sense, for $1 \leq r \leq 2$, it plays  the role  of the  Bohnenblust--Hille inequality for the case  $r=\infty$.

\begin{lemma}\label{lem: BDS}
Let $1 \leq r \leq \infty$ and  $P$ be an $m$-homogeneous polynomial in $n$ variables. Then for each $\ii \in \Jj(m-1,n)$ with associated multi-index $\alpha(\mathbf i) \in \Lambda (m-1,n)$ we have
\begin{equation}\label{BDS modif}
        \left(\sum_{k=j_{m-1}}^n |c_{(\ii,k)}(P)|^{r'}\right)^{\frac{1}{r'}} \leq em \Big(\frac{(m-1)^{m-1}}{\alpha(\mathbf i)^{\alpha(\mathbf i)}}\Big)^{\frac{1}{r}} \|P\|_{\mathcal P(^m\ell_r^n)}.
    \end{equation}
\end{lemma}

Since $\frac{(m-1)^{m-1}}{\alpha(\mathbf i)^{\alpha(\mathbf i)}} \leq e^{m-1} \vert \mathbf{i} \vert$ we immediately have
\begin{equation}\label{eq: BDS}
   \left(\sum_{k=j_{m-1}}^n |c_{(\ii,k)}(P)|^{r'}\right)^{\frac{1}{r'}} \leq m e^{1+\frac{m-1}r} |\ii|^{\frac{1}{r}} \|P\|_{\mathcal P(^m\ell_r^n)}.
\end{equation}
This is in fact the statement of \cite[Lemma~3.5.]{bayart2016monomial}. 
With it we can give the first step towards the proof of Theorem~\ref{teo: hypercontrac 1 < r <2}.

\begin{lemma}\label{lem: 1er lem 1 < r < 2}
Let $1 < r \le 2$, there is $A_{r}>0$ such that for every $m,n \in \NN$, every $P \in \Pp(^m \CC^n)$ and every decreasing $z \in \mathbb{C}^{n}$ we have
\[
 \sum_{\jj \in \Jj(m,n)} | c_\jj(P) z_\jj |  
\le A_{r} m^{1+\frac{1}{r}} e^{\frac{m}{r}}  \| z \|_{m_{\Psi_r}}^2 
\left( \sum_{k = 1}^n \frac{\log(k+1)^\frac{2}{r'}}{k^{1 + \frac{1}{r'}}} 
\sum_{\ii \in \Jj(m-2,k)} |z_\ii| |\ii|^{\frac{1}{r} } \right)  \| P \|_{\Pp(^m \ell_r^n)}.
\]
\end{lemma}
\begin{proof}
Consider $P  \in \Pp(^m \CC^n)$ as in \eqref{eternidade} and $z \in \mathbb{C}^{n}$ decreasing. Using first H\"older's inequality and then  \eqref{eq: BDS} we have 
\begin{equation}\label{desig 1er lemma 1 < r < 2}
\begin{split}
\sum_{\jj \in \Jj(m,n)} | c_\jj(P) z_\jj | 
    & = \sum_{\jj \in \Jj(m-1,n)}  \sum_{j_m = j_{m-1}}^n | c_{(\jj,j_m)}(P) z_\jj z_{j_m} | \\
	& \le \sum_{\jj \in \Jj(m-1,n)} |z_\jj| \Big( \sum_{j_m = j_{m-1}}^n |c_{(\jj,j_m)}(P)|^{r'} \Big)^{\frac{1}{r'}}  \Big( \sum_{j_m = j_{m-1}}^n |z_{j_m}|^{r} \Big)^{\frac{1}{r}}\\
	& \le e^{1-\frac{1}{r}} m e^{\frac{m}{r}} \| P \|_{\Pp(^m \ell_r)} \sum_{\jj \in \Jj(m-1,n)} |z_\jj| |\jj|^{\frac{1}{r}} \Big( \sum_{ j_m = j_{m-1}}^n |z_{j_m}|^{r} \Big)^{\frac{1}{r}} \\
	& =  e^{1-\frac{1}{r}} m e^{\frac{m}{r}} \| P \|_{\Pp(^m \ell_r)} \sum_{j_{m-1}=1}^n |z_{j_{m-1}}| \sum_{\ii \in \Jj(m-2,j_{m-1})} |z_\ii| |(\ii,j_{m-1})|^{\frac{1}{r}} \Big( \sum_{ j_m = j_{m-1}}^n |z_{j_m}|^{r} \Big)^{\frac{1}{r}}  \\
    &  \le e^{1-\frac{1}{r}} m e^{\frac{m}{r}} \| P \|_{\Pp(^m \ell_r)} (m-1)^{\frac{1}{r}} \sum_{j_{m-1}=1}^n |z_{j_{m-1}}| \Big( \sum_{j_m = j_{m-1} }^n |z_{j_m}|^r \Big)^{\frac{1}{r}} \sum_{\ii \in \Jj(m-2,j_{m-1})} |z_\ii| |\ii|^\frac{1}{r} ,
\end{split}
\end{equation}
where the last inequality is due to the fact that $|(\ii,j_{m-1})| \le (m-1) |\ii|$ for every $ \ii \in \Jj(m-2, j_{m-1})$.

We now bound the factor $|z_{j_{m-1}}| \Big( \sum_{j_m = j_{m-1} }^n |z_{j_m}|^r \Big)^{\frac{1}{r}} $. For each $1 \leq j \leq n$ we use Remark~\ref{rem: desig coord} to obtain (note that $\frac{r}{r'}-1=r-2 \leq 0$).
\begin{multline*}
    |z_{j}| \Big( \sum_{k = j}^n |z_{k}|^r \Big)^{\frac{1}{r}} 
    \le  \| z \|_{m_{\Psi_r}}^2 \frac{\log(j+1)^\frac{1}{r'}}{j} 
\Big( \sum_{k = j}^n \frac{\log(k+1)^{\frac{r}{r'}}}{k^r}  \Big)^{\frac{1}{r}} \\
    \le \| z \|_{m_{\Psi_r}}^2 \frac{\log(j+1)^\frac{1}{r'}}{j} 
\log(j+1)^{\frac{1}{r'} - \frac{1}{r}} 
\Big( \sum_{k = j}^n \frac{\log(k+1)}{k^r}  \Big)^{\frac{1}{r}} \,.
    \end{multline*}
We deal with the last sum
\begin{multline*}
    \sum_{k= j}^n \frac{\log(k+1)}{k^r} \leq \Big( 1 + \frac{1}{j} \Big)^{r} \sum_{k= j}^n \frac{\log(k+1)}{(k+1)^r} 
\leq 2^{r} \sum_{k = j+1 }^{n+1} \frac{\log(k)}{k^r} 
\leq 2^{r+2} \int_{j}^{n +1}\frac{\log(x)}{x^r} dx 
\\
\leq  2^{r+2} \frac{(r-1) \log(j) +1}{(r-1)^{2} j^{r-1}}
\leq 2^{r+2} \frac{2r}{(r-1)^{2}}  \frac{\log(j+1)}{j^{r-1}}   \,,
\end{multline*}
and
\[
    |z_{j}| \Big( \sum_{k = j}^n |z_{k}|^r \Big)^{\frac{1}{r}}  
\le 2^{r+2} \frac{2r}{(r-1)^{2}}  \| z \|_{m_{\Psi_r}}^2 \frac{\log(j+1)^\frac{2}{r'}}{j^{1+\frac{1}{r'}}} 
\]
This and \eqref{desig 1er lemma 1 < r < 2} give the conclusion
\end{proof} 

In view of Lemma~\ref{lem: 1er lem 1 < r < 2}, now we need to bound $\sum_{\ii \in \Jj(m-2,k)} |z_\ii| |\ii|^{\frac{1}{r} }$ in a suitable way (depending on $k$). To this purpose we switch to the $\alpha$-notation of multi-indices (recall \eqref{eternidade}), that is going to be more convenient. Then the sum reads
\begin{equation} \label{pontes}
\sum_{\alpha \in \Lambda (m-2,k)} \vert z \vert^{\alpha} \vert [\alpha] \vert
\end{equation}
and the strategy is to decompose this sum into two sums: a tetrahedral and an even part and, then, bound each one of these.  This lies in the general philosophy of decomposing index sets into some smaller subset in which a certain problem results easier and, at the same time, are the bricks in which any general index can be recovered. This philosophy has alredy been used in \cite{galicer2017mixed}.\\

Let us be more precise and introduce some notation. A multi-index $\alpha$ is tetrahedral if all its entries are either $0$ or $1$. We consider the set  of tetrahedral multi-indices
\[
\Lambda_T(m,n) = \big\lbrace \alpha \in \Lambda(m,n) : \alpha_i \in \{0,1\}  \big\rbrace.
\]
A multi-index is called even if all its non-zero entries are even (note that this forces the multi-index to have even order). We consider then the set 
\[
\Lambda_E(m,n) = \big\lbrace \alpha \in \Lambda(m,n) :  \alpha_i \text{ is even for every } i=1,\ldots,n \big\rbrace.
\]
Observe that for every $\alpha \in \Lambda_E(m,n)$ there is a unique $\beta \in \Lambda(m/2,n)$ such that $\alpha = 2 \beta$.

\begin{remark}\label{rem: fact indices}
Given $\alpha \in \Lambda (M,N)$ define $\alpha_{T}$ (the tetrahedral part) and $\alpha_{E}$ (the even part) as
\[
\big(\alpha_{T} \big)_{i} = \begin{cases}
1 & \text{ if } \alpha_{i} \text{ is odd} \\
0 & \text{ if } \alpha_{i} \text{ is even} 
\end{cases}
\quad \text{ and } \quad
\big(\alpha_{E} \big)_{i} = \begin{cases}
\alpha_{i}-1 & \text{ if } \alpha_{i} \text{ is odd} \\
\alpha_{i} & \text{ if } \alpha_{i} \text{ is even} 
\end{cases}
\,.
\]
If $0 \leq k \leq M$ is the number of odd entries in $\alpha$, then clearly $\alpha_{T} \in \Lambda_T(k,N)$  and $\alpha_E \in \Lambda_E(M-k,N)$ and $\alpha = \alpha_T + \alpha_E$. 
As $(\alpha_E)_i \le \alpha_i$ for every $i$ then $\alpha_E! \le \alpha!$. On the other hand, $\alpha_T! = 1$, then $\alpha_T! \alpha_E! \le \alpha!$, and
\[
|[\alpha]| = \frac{M!}{\alpha!} \le \frac{M!}{\alpha_T! \alpha_E!} = \frac{M!}{(M-k)! k!} \frac{k!}{\alpha_T!} \frac{(M-k)!}{\alpha_E!} = \left( \genfrac{}{}{0pt}{}{M}{k} \right) |[\alpha_T]||[\alpha_E]|  
\le 2^M |[\alpha_T]||[\alpha_E]|.
\]
\end{remark}

Our next step is to bound a sum as in \eqref{pontes} when we just consider even or tetrahedral indices. We start with the latter.

\begin{lemma}\label{lem: tetra 1 < r < 2}
For every $1 < r \le 2$ and $M,N \in \NN$, and every decreasing $z \in \mathbb{C}^{N}$ we have 
\[
\sum_{\alpha \in \Lambda_T(M,N)} |z^\alpha| |[\alpha]|^{\frac{1}{r}} \le  2 (1+\varepsilon)^{\frac{M}{r'}} 
\| z\|_{m_{\Psi_r}}^M N^{\frac{1}{(1 +\varepsilon)r'}},
\]
for every $\varepsilon > 0$ and
\[
\sum_{\alpha \in \Lambda_E(M,N)} |z^\alpha| |[\alpha]|^{\frac{1}{r}} 
\le \| z \|_{\ell_r}^M \le  \| \id : m_{\Psi_r} \to \ell_r \|^M \| z\|_{m_{\Psi_r}}^M.
\]
\end{lemma}

\begin{proof}
We begin with the first inequality, observing that it is obvious if $N=1$. We may, then, assume $N \geq 2$. Then, given $\alpha \in \Lambda_T(M,N)$, note that $\alpha !=1$ and $|[\alpha]|$ is exactly $M!$. Then,  
\begin{multline*}
\sum_{\alpha \in \Lambda_T(M,N)} |z^\alpha| |[\alpha]|^{\frac{1}{r}}
	 = \sum_{\alpha \in \Lambda_T(M,N)} |z^\alpha| |[\alpha]| \frac{1}{|[\alpha]|^{\frac{1}{r'}}} 
     \le \Big( \sum_{k=1}^N |z_k| \Big)^{M} \frac{1}{M!^{\frac{1}{r'}}} \\
     \le \| z \|_{m_{\Psi_r}}^M \log(N+1)^{\frac{M}{r'}} \frac{1}{M!^{\frac{1}{r'}}} 
     \le 2 \| z \|_{m_{\Psi_r}}^M \Big( \frac{\log(N)^{M}}{M!}\Big)^{\frac{1}{r'}}  \,.
\end{multline*}
A simple calculus argument shows that  the function $f:[1,\infty[  \rightarrow \RR$ given by $f(x)= \frac{\log(x)^{M}}{x^{1/(1 + \varepsilon)}}$ is bounded by $\big(\frac{(1+\varepsilon)M}{e} \big)^{M}$, then $ \log(N)^M \le N^{1/(1+\varepsilon)}  \big(\frac{(1+\varepsilon)M}{e} \big)^{M}$ . On the other hand 
$M! \ge \left( \frac{M}{e} \right)^M$.This gives the conclusion.\\
For the proof of the second inequality let us recall first that for each $\alpha \in \Lambda_E(M,N)$ there is a unique $\beta \in {\Lambda}(M/2,N)$ such that $\alpha = 2 \beta$ and, moreover,
\[
|[\alpha]| = \frac{M!}{\alpha_1 ! \cdots \alpha_N!} = \Big(\frac{(M/2)!}{\beta_1 ! \cdots \beta_N!}\Big)^2 \frac{M!}{(M/2)!(M/2)!} \prod_{i=1}^N \frac{\beta_i! \beta_i!}{(2\beta_i)!} \le |[\beta]|^2,
\]
where last inequality holds because $2^k \le \frac{(2k)!}{k!^2} \le 2^{2k}$ and then
\[
\frac{M!}{(M/2)!(M/2)!} \prod_{i=1}^N \frac{\beta_i! \beta_i!}{(2\beta_i)!} \le 2^M \prod_{i=1}^N \frac{1}{2^{\beta_i}} = 1.
\]
Then (note that, since $2/r \ge 1$, the $\ell_1$ norm bounds the $\ell_{2/r}$ norm)
\begin{multline*}
\sum_{\alpha \in \Lambda_E(M,N)} |z^\alpha| |\alpha|^{\frac{1}{r}} 
	 \le \dis\sum_{\beta \in {\Lambda}(M/2,N)} |(z^2)^\beta| |\beta|^{2/r} 
     =  \dis\sum_{\beta \in {\Lambda}(M/2,N)} \Big( |(z^r)^\beta| |\beta| \Big)^{2/r}\\
     \le \Big( \dis\sum_{\beta \in {\Lambda}(M/2,N)} |(z^r)^\beta| |\beta| \Big)^{2/r} 
     =  \Big( \sum_{l=1}^N |z_l|^r \Big)^{M/r} 
     \le \|\id : m_{\Psi_r} \to \ell_r \|^M  \| z \|_{m_{\Psi_r}}^M \,.
\end{multline*}
\end{proof}

\begin{lemma}\label{lem: general}
Given $1 < r \le 2$ there is a constant $K_{r}\geq 1$ such that for every  $M,N \in \NN$, and every decreasing $z \in \mathbb{C}^{N}$ we have 
\[
\sum_{\alpha \in \Lambda(M,N)} |z^\alpha| |[\alpha]|^{\frac{1}{r}} 
\le K_{r} (M+1) (1+\varepsilon)^{\frac{M}{r'}} 2^{\frac{M}{r}+1} N^{\frac{1}{(1+\varepsilon)r'}} \|\id : m_{\Psi_r} \to \ell_r \|^M \| z\|_{m_{\Psi_r}}^M,
\]
for every $\varepsilon>0$.
\end{lemma}
\begin{proof}
Choose some decreasing $z$ and use  by Remark~\ref{rem: fact indices} and Lemma~\ref{lem: tetra 1 < r < 2} to get
\begin{align*}
\sum_{\alpha \in \Lambda(M,N)} |z^\alpha| |[\alpha]|^{\frac{1}{r}}
& =  \sum_{k = 0}^M \sum_{\alpha_{T} \in \Lambda_T(k,N)} \sum_{\alpha_{E} \in \Lambda_E(M-k,N)}|z^{(\alpha_{T} + \alpha_{E})}| |[\alpha_{T} + \alpha_{E}]|^\frac{1}{r}\\
& \leq 2^{\frac{M}{r}} \sum_{k = 0}^M 
\left( \sum_{\alpha_{T} \in \Lambda_T(k,N)} |z^\alpha_{T} | | [\alpha_{T}]|^{\frac{1}{r}} \right) 
\left( \sum_{\alpha_{E} \in \Lambda_E(M-k,N)}|z^\alpha_{E}| |[\alpha_{E}]|^\frac{1}{r} \right) \\
& \leq 2^{\frac{M}{r}} \sum_{k = 0}^M 
\left((1+\varepsilon)^{\frac{k}{r'}} \| z \|_{m_{\Psi_r}}^k N^{\frac{1}{(1+\varepsilon)r'}} \right) 
\left( \| \id: m_{\Psi_r} \to \ell_r  \|^{M-k}  \| z \|_{m_{\Psi_r}}^{M-k} \right) \\
& \leq 2^{\frac{M}{r}+1} (1 + \varepsilon)^M \| \id: m_{\Psi_r} \to \ell_r  \|^{M} \| z \|_{m_{\Psi_r}}^{M} N^{\frac{1}{(1 + \varepsilon)r'}}  \sum_{k = 0}^M  2^{k(1-\frac{2}{r})} \,.
\end{align*}
For $r=2$ the last sum is exactly $M+1$. If $1<r<2$ the series converges to $\frac{2^{2/r}}{2^{2/r} - 2}$. This completes the proof
\end{proof}

We are finally in the position to give the proof of Theorem~\ref{teo: hypercontrac 1 < r <2} from which (as we already saw) the lower inclusion in Theorem~\ref{thm: main theorem hb} follows.

\begin{proof}[Proof of Theorem~\ref{teo: hypercontrac 1 < r <2}]
Fix  $1 < r \le 2$ and $n,m$. Pick then $P \in P \in \Pp(^m \CC^n)$ and $z \in \mathbb{C}^{n}$. Since $\Vert z \Vert_{m_{\Psi_r}} = \Vert z^{*} \Vert_{m_{\Psi_r}}$, we may assume $z=z^{*}$. Applying  Lemma~\ref{lem: general} with $M = m-2$ and $N = k$ after Lemma~\ref{lem: 1er lem 1 < r < 2} yields
\begin{multline*}
\sum_{\jj \in \Jj(m,n)} | c_\jj(P) z_\jj | \\
\leq 2 A_{r} m^{1+\frac{1}{r}} e^{\frac{m}{r}} \Vert z \Vert_{m_{\Psi_r}}^{2}
\left( \sum_{k = 1}^n \frac{\log(k+1)^\frac{2}{r'}}{k^{1 + \frac{1}{r'}}} K_{r} (m-1) 2^{\frac{(m-2)}{r}}(1+\varepsilon)^{\frac{m-2}{r'}} \| \id \|^{m-2} k^{\frac{1}{(1+\varepsilon)r'}} \Vert z \Vert_{m_{\Psi_r}}^{m-2} \right)  \| P \|_{\Pp(^m \ell_r^{n})} \\
\leq 2 A_{r}K_{r} m^{2+\frac{1}{r}} ((1+\varepsilon)2e)^{\frac{m}{r}}\| \id \|^{m} \Vert z \Vert_{m_{\Psi_r}}^{m}
\left( \sum_{k = 1}^n \frac{\log(k+1)^\frac{2}{r'}}{k^{1 + \frac{\varepsilon}{(1+\varepsilon)r'}}}    \right)  \| P \|_{\Pp(^m \ell_r^{n})} \,.
\end{multline*}
Since $r>1$ the series $\sum_{k = 1}^\infty \frac{\log(k+1)^\frac{2}{r'}}{k^{1 + \frac{\varepsilon}{(1+\varepsilon)r'}}}$ is convergent. This completes the proof.
\end{proof}

\section{Monomial convergence for bounded holomorphic functions on \texorpdfstring{$B_{\ell_r}$}.}

We change now our focus to the space $H_{\infty} (B_{\ell_{r}})$ of bounded holomorphic functions on $B_{\ell_r}$. Our main contribution in this side is the following theorem, that provides with lower and upper inclusions for the set of 
monomial convergence of these spaces. It recovers (see Remark~\ref{georgewinston} and Corollary~\ref{stabatmater}) some previously known results.

\begin{theorem}\label{teo: mon H_infty}
Let $1 <r \leq 2$ then, 
\begin{multline*}
\Big\{ z \in \mathbb{C}^{\mathbb{N}} \colon  2e \| \id: m_{\Psi_r} \to \ell_r \|^r \left(  \limsup_{n \to \infty} \frac{\sum_{k=1}^n z_k^*}{\log(n+1)^{1-1/r}}\right)^r  + \| z \|_{\ell_r}^r < 1  \Big\} 
\subset \\
\mon H_{\infty} (B_{\ell_{r}}) \subset 
\Big\{ z \in B_{\ell_r} \colon \limsup_{n \to \infty} \frac{\sum_{k=1}^n z_k^*}{\log(n+1)^{1-1/r}} \leq 1 \Big\} \,.
\end{multline*}
\end{theorem}

The upper inclusion follows using probabilistic techniques, as in the case of $\mon H_{b}(\ell_{r})$. The lower inclusion, on the other hand, relies on Theorem~\ref{teo: hypercontrac 1 < r <2} and requires some preliminary work that we 
start with the following remark.

\begin{remark}\label{relacion mon H_infty}
Given a Reinhardt domain $\mathcal{R}$ in a Banach sequence space $X$, a simple closed-graph argument (see \cite[Lemma~4.1]{defant2009domains} or \cite[Remark~20.1]{defant2019libro}) shows that $z \in \mon H_{\infty} (\mathcal{R})$ if and only if there is a constant $C_{z}>0$ such that 
\[
\sum_{\alpha\in \mathbb N_0^{(\mathbb N)}}  |c_{\alpha}(f) z^\alpha| \le C_z \| f \|_{\mathcal{R}}
\]
for every $f \in H_{\infty} (\mathcal{R})$.
\end{remark}

\begin{lemma}\label{lem: 1er result H_infty}
Let $1< r \le 2$ then, $\frac{1}{\| \id: m_{\Psi_r} \to \ell_r \|(2e)^{1/r}}B_{m_{\Psi_r}} \subset \mon H_\infty (B_{\ell_r})$.
\end{lemma}

\begin{proof}
In order to keep thinsg readable we write $K= \| \id: m_{\Psi_r} \to \ell_r \|(2e)^{1/r}$. We first show that if $z \in \frac{1}{K} B_{m_{\Psi_r}}$ is non-decreasing, then $z \in \mon H_\infty (B_{\ell_r})$. The general result follows 
form the fact that $B_{m_{\Psi_r}}$ and $\mon H_\infty (B_{\ell_r})$ are both symmetric (Corollary~\ref{cor: familias de reordenamiento}). We choose now $f \in H_{\infty} (B_{\ell_r})$ and fix $\varepsilon >0$ so that 
$(1+\varepsilon)^{1/r} \Vert z \Vert_{m_{\Psi_{r}}} K <1$. By Theorem~\ref{teo: hypercontrac 1 < r <2} we can find $C_{r}(\varepsilon)>0$ so that 
\begin{align*}
\sum_{\alpha \in \mathbb{N}_0^{(\mathbb{N})} } \vert  c_\alpha(f) z^\alpha \vert & = \sup_{n \in \mathbb{N}} \sum_{m = 0}^\infty \sum_{\jj \in \mathcal J(m,n) } \vert c_{\jj}(f) z_{\jj} \vert \\
&  \leq \sup_{n \in \mathbb{N}} \sum_{m = 0}^\infty  C_r(\varepsilon) m^{2+\frac{1}{r}} (1+\varepsilon)^{\frac{m}{r}} K^{m} \| z \|_{m_{\Psi_r}}^m \sup_{u \in B_{\ell_{r}^{n}}} \left\vert \sum_{\jj \in \mathcal J(m,n) } c_{\jj}(f) u_{\jj} \right\vert \\
& \leq \sum_{m = 0}^\infty  C_r(\varepsilon) \left(m^{\frac{1}{m}(2+\frac{1}{r})} (1+\varepsilon)^{\frac{1}{r}} K \| z \|_{m_{\Psi_r}} \right)^m \| P_m(f) \|_{\Pp(^m \ell_r)} \\
& \le \| f \|_{B_{\ell_r}}  C_r(\varepsilon) \sum_{m = 0}^\infty \left(m^{\frac{1}{m}(2+\frac{1}{r})} (1+\varepsilon)^{\frac{1}{r}} K \| z \|_{m_{\Psi_r}} \right)^m.
\end{align*}
The choice of $\varepsilon$ and fact that $m^{\frac{1}{m}(2+\frac{1}{r})} \to 1$ as $m \to \infty$ immediately give that the series converges and complete the proof.
\end{proof}

A useful tool when dealing with $\mon H_{\infty} (B_{c_{0}})$ is that, if a sequence belong to such a  set of monomial convergence and we modify finitely many coordinates, then the resulting sequence remains in the set of monomial 
convergence (see  \cite[Lemma~2]{defant2008bohr} or \cite[Proposition~10.14]{defant2019libro}).  It is unknown whether or not an analogous result result holds for $\ell_r$ (see the comments regarding this problem in 
\cite[Chapter~10]{schluters2015unconditionality}). We overcome this with the following proposition, a weaker version of this, but enough for our purposes.

\begin{proposition}\label{prop: achicar coord}
Let $1< r < \infty$ and $u,z \in B_{\ell_r}$ be such that $ |u_n| \le |z_n|$ for $1 \le n \le N$ and $|u_n| = |z_n|$ for $n > N$. Suppose that there exists $\rho > \sum_{n=1}^{N}|z_n|^r$ so that $u \in \mon H_{\infty}((1-\rho)^{1/r} B_{\ell_r})$. Then $z \in \mon H_{\infty} (B_{\ell_r})$.
\end{proposition}

\begin{proof}
Let $a_1, \ldots, a_{N} $  be positive real numbers such that $|z_i| < a_i$ for every $1 \le i \le N$ and 
\[
a := \sum_{n = 1}^{N} a_n^r < \rho.
\]
Given for $f \in H_{\infty}(B_{\ell_r})$ and $k_1, \ldots, k_{N} \in \NN$, we define (following the proof of  \cite[Lemma~2]{defant2008bohr})
\[
    f_{k_1, \ldots, k_{N}}(\nu) := \frac{1}{(2\pi i)^{N}} \int_{|w_1| = a_1} \cdots \int_{|w_{N}| = a_{N}} \frac{f(w_1, \ldots, w_{N}, \nu_{N+1}, \nu_{N+2}, \ldots)}{w_1^{k_1+1} \cdots w_{N}^{k_{N}+1}} dw_1 \cdots dw_{N}. 
\]
Note that $f_{k_1, \ldots,k_N}$ is well defined on the contracted ball $(1-a)^{1/r}B_{\ell_r}$ and, in fact, belongs to $H_{\infty}((1-a)^{1/r} B_{\ell_r})$ (because $f \in H_{\infty}(B_{\ell_r})$) and
\begin{equation}\label{cota f_(k_1,...k_n)}
\| f_{k_1, \ldots,k_{N}}\|_{(1-a)^{1/r} \cdot B_{\ell_r} } \le 
\frac{\| f \|_{B_{\ell_r}}}{a_1^{k_1} \cdots a_N^{k_N}}.
\end{equation}
Our next step is to understand the coefficients $c_\alpha(f_{k_1, \ldots, k_N})$ in relation to those of $f$. For each multi-index $\alpha = (\alpha_1, \ldots, \alpha_n, 0 , \ldots)$ with $\alpha_n \neq 0$, an application of the Cauchy integral formula yields
\begin{equation}\label{coeff de f_(k1,...,kN)}
c_{\alpha}(f_{k_1,\ldots,k_n}) = 
\begin{cases}
c_{(k_1, \ldots,k_N, \alpha_{N+1}, \ldots , \alpha_n)}(f) & \; \text{if} \; \alpha_1 = \cdots = \alpha_N = 0,\\
0 & \text{otherwise}.
\end{cases}
\end{equation}
We have now everything we need to proceed. Note that, since $a < \rho$, we have $u \in \mon H_\infty((1-\rho)^{1/r}B_{\ell_r}) \subset \mon H_\infty ((1-a)^{1/r}B_{\ell_r}) $. With Remark~\ref{relacion mon H_infty} and  \eqref{cota f_(k_1,...k_n)} we get
\begin{equation}\label{eq: u en mon de f_(k1,...,kN)}
    \sum_{\beta \in \NN_0^{(\NN)}} |c_\beta(f_{k_1, \ldots,k_N})| |u_{N+1}^{\beta_1} \cdots u_{N+2}^{\beta_2} \cdots| \le C_u \| f_{k_1,\ldots,k_N} \|_{(1-a)^{1/r}B_{\ell_r}} \le C_u \frac{\| f \|_{B_{\ell_r}}}{a_1^{k_1} \cdots a_N^{k_N}}.
\end{equation}
Now using \eqref{coeff de f_(k1,...,kN)} and \eqref{eq: u en mon de f_(k1,...,kN)} (recall that $\vert u_{n} \vert = \vert z_{n} \vert$ for $n \geq N+1$) we have
\begin{align*}
\sum_{\alpha \in \NN_0^{(\NN)}} |c_\alpha(f)||z^\alpha| 
    & = \sum_{(k_1, \ldots k_N) \in \NN_0^N} |z_1^{k_1} \cdots z_N^{k_N}| \sum_{\beta \in \NN_0^{(\NN)}} |c_{(k_1, \ldots,k_N,\beta)}(f)| |u_{N+1}^{\beta_1} \cdots u_{N+2}^{\beta_2} \cdots| \\
    & = \sum_{(k_1, \ldots k_N) \in \NN_0^N} |z_1^{k_1} \cdots z_N^{k_N}| \sum_{\beta \in \NN_0^{(\NN)}} |c_\beta(f_{k_1, \ldots,k_N})| |u_{N+1}^{\beta_1} \cdots u_{N+2}^{\beta_2} \cdots| \\
    & \le \sum_{(k_1, \ldots k_N) \in \NN_0^N} |z_1^{k_1} \cdots z_N^{k_N}| C_u \frac{\| f \|_{B_{\ell_r}}}{a_1^{k_1} \cdots a_N^{k_N}} \\
    & = C_u \| f \|_{B_{\ell_r}} \prod_{n = 1}^N \sum_{k_n \ge 0} \left( \frac{|z_n|}{a_n} \right)^{k_n} < \infty,
\end{align*}
as we wanted.
\end{proof}

Let us make a last observation before we proceed with the proof of Theorem~\ref{teo: mon H_infty}. Given a Banach sequence space $X$, for every $f \in H_{\infty}(tB_{X})$ and $t > 0$ the function $f_{t}$ given by $f_{t}(x)= f(tx)$ for $x \in B_{X}$ belongs to $H_{\infty}(B_{X})$ and 
$c_\alpha (f_t) = t^{|\alpha|} c_\alpha(f)$ for every $\alpha$. Then, if $z \in \mon H_{\infty}(B_{X})$ we have
\[
   \sum_{\alpha \in \NN_0^{(\NN)}} |c_\alpha(f) (t z)^{\alpha}| 
    = \sum_{\alpha \in \NN_0^{(\NN)}} |c_\alpha(f)t^{|\alpha|} z^{\alpha}| 
    = \sum_{\alpha \in \NN_0^{(\NN)}} |c_\alpha(f_{t}) z^{\alpha}| < \infty.
\]
This implies $t \mon H_\infty(B_X) \subset \mon H_\infty (t B_x)$ for every Banach sequence space $X$ and every $t>0$.\\
Noting that $tB_{X}$ is the open unit ball of the Banach sequence space $(X, t \Vert \cdot \Vert_{X})$, the previous inclusion yields
\[
t^{-1} \mon H_\infty (t B_X) \subset \mon H_\infty (t^{-1} t B_X) = \mon H_\infty (B_X).
\]
This altogether shows
\begin{equation}\label{eq: dilatado mon}
    \mon H_\infty (t B_X) = t \mon H_\infty (B_X)
\end{equation}
for every Banach sequence space $X$ and every $t>0$.
We are now in conditions of proving Theorem~\ref{teo: mon H_infty}.

\begin{proof}[Proof of Theorem~\ref{teo: mon H_infty}]
Let us start with the upper inclusion 
\[ \mon H_{\infty} (B_{\ell_{r}}) 
\subset \Big\{ z \in B_{\ell_r} \colon \limsup_{n \to \infty} \frac{\sum_{k=1}^n z_k^*}{\log(n+1)^{1-1/r}} \leq 1 \Big\} .
\]
Fix $z \in \mon H_\infty (B_{\ell_{r}})$. Arguing as in the proof of the upper inclusion of Theorem~\ref{thm: main theorem hb}, proceeding as in \eqref{elenitagentil}, replacing the role of Lemma~\ref{relacion mon} by Remark~\ref{relacion mon H_infty}, and as in \eqref{desigBayart} we get
\[
 \sum_{j = 1}^n |z_j^*| \le 
  C_{z^*,r}^{\frac{1}{m}} \left[\log(m)^{\frac{1}{m}} (2 \pi m)^{\frac{1}{2m}} e^{\frac{1}{12m^2}}  \frac{m}{e} n^{\frac{1}{m}}\right]^{1 - \frac{1}{r}}. 
 \]
where $C_{z^*,r}$ is a positive constant that depends only on $z^*$ and $r$.
Choosing $m = \lfloor \log(n+1) \rfloor $ we get
\[
\limsup_{n \to \infty} \frac{1}{\log(n+1)^{1-\frac{1}{r}}} \sum_{k = 1}^n |z_n^*| \le 1,
\]
which gives our claim.\\
We now face the proof of the lower inclusion 
\[
\Big\{ z \in \mathbb{C}^{\mathbb{N}} \colon  2e \| \id: m_{\Psi_r} \to \ell_r \|^r \left(  \limsup_{n \to \infty} \frac{\sum_{k=1}^n z_k^*}{\log(n+1)^{1-1/r}}\right)^r  + \| z \|_{\ell_r}^r < 1  \Big\} \subset \mon H_{\infty} (B_{\ell_{r}}) .
\]
In order to keep the notation as simple as possible, let $K= 2e \| \id: m_{\Psi_r} \to \ell_r \|^r$. Take $z \in \mathbb{C}^{\mathbb{N}}$ such that 
\[
  K \left(  \limsup_{n \to \infty} \frac{\sum_{k=1}^n z_k^*}{\log(n+1)^{1-1/r}}\right)^r  + \| z \|_{\ell_r}^r < 1,
\]
and note that this implies $z \in B_{\ell_{r}}$. Denote $L:= \dis\limsup_{n \to \infty} \frac{\sum_{k=1}^n z_k^*}{\log(n+1)^{1-1/r}} $, choose $\varepsilon > 0$ so that 
\begin{equation}\label{boulanger}
K \big(  (1+\varepsilon) L \big)^r  + \| z \|_{\ell_r}^r < 1,
\end{equation}
and $N \in \NN$ for which  
\[
\sup_{ n \ge N} \frac{\sum_{k=1}^n z_k^*}{\log(n+1)^{1-1/r}} < (1+\varepsilon) L.
\]
Let us observe that 
\begin{equation} \label{debussy}
z_{N}^{*} < \frac{\log(N+1)^{1-1/r}}{N} (1+\varepsilon) L,
\end{equation}
(this follows essentially as in Remark \ref{rem: desig coord}) and define $u = (\underbrace{z_{N}^{*}, \ldots, z_{N}^{*}}_{N}, z_{N +1}^{*}, z_{N +2}^{*},  \ldots)$. On the one hand, for every $n < N$ we have, using \eqref{debussy},
\[
\frac{\sum_{k=1}^n u_k^*}{\log(n+1)^{1-1/r}} 
< (1+\varepsilon) L .
\]
On the other hand, for $n \ge N$,
\[
\frac{\sum_{k=1}^n u_k^*}{\log(n+1)^{1-1/r}} =
\frac{\sum_{k=1}^n z_{n}}{\log(n+1)^{1-1/r}} < (1+\varepsilon) L.
\]
This altogether gives $\| u \|_{m_{\Psi_r}} < (1+\varepsilon) L$. We choose $\rho > \sum_{k=1}^{N} |z_k|^r$ such 
\[
\| \id: m_{\Psi_r} \to \ell_r \|^r (2e) (L(1+\varepsilon))^r  + \rho < 1,
\]
and, using \eqref{boulanger} we get
\[
\| u \|_{m_{\Psi_r}} < (1+\varepsilon) L < \frac{(1 - \rho)^{1/r}}{ \| \id: m_{\Psi_r} \to \ell_r \|(2e)^{1/r} }.
\]
Lemma~\ref{lem: 1er result H_infty} and equation \eqref{eq: dilatado mon} imply $u \in \mon H_\infty((1 - \rho)^{1/r} B_{\ell_r})$ and, then Proposition~\ref{prop: achicar coord} gives $z^{*} \in \mon H_{\infty} (B_{\ell_{r}})$. Finally, Corollary~\ref{reordenamiento} yields $z  \in \mon H_{\infty} (B_{\ell_{r}})$ and completes the proof.
\end{proof}

\begin{remark} \label{georgewinston}
Theorem~\ref{teo: mon H_infty} implies other known results which try to characterize the set of monomial convergence of $H_\infty(B_{\ell_r})$. Note first that, if $z \in \ell_{1}$, then 
\[
\limsup_{n \to \infty} \frac{\sum_{k=1}^n z_k^*}{\log(n+1)^{1-1/r}} = 0 .
\]
Thus
\[
B_{\ell_r} \cap \ell_1 \subset \Big\{ z \in \mathbb{C}^{\mathbb{N}} \colon  2e \| \id: m_{\Psi_r} \to \ell_r \|^r \left(  \limsup_{n \to \infty} \frac{\sum_{k=1}^n z_k^*}{\log(n+1)^{1-1/r}}\right)^r  + \| z \|_{\ell_r}^r < 1  \Big\}.
\]
On the other hand, if $z \in B_{\ell_{r}}$ is such that 
\[
\limsup_{n \to \infty} \frac{\sum_{k=1}^n z_k^*}{\log(n+1)^{1-1/r}} \leq 1
\]
then there is a constant $c>0$ so that
\[
z_{n}^{*} \leq c \frac{\log(n+1)^{1-1/r}}{n}.
\]
From this we easily get that $z \in \ell_{1 + \varepsilon}$ for every $\varepsilon >0$, and we recover \eqref{saintsaens} from Theorem~\ref{teo: mon H_infty}.
\end{remark}

The following corollary extends \cite[Theorem~5.5(1a) and Corollary~5.7]{bayart2016monomial} for $1 < r \le 2$.

\begin{corollary} \label{stabatmater}
Let $1 < r \le 2$. Then
\begin{equation} \label{pessoa}
\left( \frac{1}{n^{1/r'} \log(n+2)^\theta}\right)_{n \ge 1} \cdot B_{\ell_r} \subset \mon H_{\infty} (B_{\ell_r})
\end{equation}
for every  $\theta >0$. Also, denoting $K = \frac{1}{(2e \| \id: m_{\Psi_r} \to \ell_r \| +1)^{1/r}}$, we have
\begin{equation} \label{reis}
\left( \frac{1}{K n^{1/r'} } \right)_{n \ge 1} \cdot B_{\ell_r} \subset \mon H_{\infty} (B_{\ell_r}) .
\end{equation}
\end{corollary}

\begin{proof}
Let us begin by proving \eqref{pessoa}. Fix $\theta >0$ and choose $z \in \left( \frac{1}{n^{1/r'} \log(n+2)^\theta}\right)_{n \ge 1} B_{\ell_r}$. We can find $w \in B_{\ell_r}$ so that $z_n = \frac{w_n}{n^{1/r'}\log(n + 1)^\theta}$
for every $n \in \mathbb{N}$. Since $z \in c_{0}$, there is an injective $\sigma : \NN \to \NN$ such that $z_n^* = |z_{\sigma(n)}| = \frac{|w_{\sigma(n)}|}{\sigma(n)^{1/r'} \log(\sigma(n)+2)^\theta}$. Using H\"older's inequality we get
\begin{align*}
    \frac{1}{\log(n+1)^{1/r'}} \sum_{l = 1}^n  z_l^* 
    & = \frac{1}{\log(n+1)^{1/r'}}  \sum_{l = 1}^n \frac{|w_{\sigma(l)}|}{\sigma(l)^{1/r'} \log(\sigma(l)+2)^\theta} \\
    & \le \frac{1}{\log(n+1)^{1/r'}}  \left( \sum_{l = 1}^n |w_{\sigma(l)}|^r \right)^{1/r}  \left( \sum_{l = 1}^n\frac{1}{\sigma(l) \log(\sigma(l)+2)^{r'\theta}} \right)^{1/r'} \\
    & \le \frac{1}{\log(n+1)^{1/r'}} \left( \sum_{l = 1}^n\frac{1}{\sigma(l) \log(\sigma(l)+2)^{r'\theta}} \right)^{1/r'} \\
    & \le \frac{1}{\log(n+1)^{1/r'}} \left( \sum_{l = 1}^n\frac{1}{l \log(l+2)^{r'\theta}} \right)^{1/r'} ,
\end{align*}
where the last inequality holds because $x \mapsto \frac{1}{x \log(x+2)^{r'\theta}}$ defines a decreasing function for $x >1$. 
The last term, $\frac{1}{\log(n+1)^{1/r'}} \left( \sum_{l = 1}^n\frac{1}{l \log(l+2)^{r'\theta}} \right)^{1/r'}$, goes to $0$ as $n \to \infty$, and therefore
\[
\lim\sup_{n \to \infty} \frac{1}{\log(n+1)^{1/r'}} \sum_{l = 1}^n  z_l^* = 0.
\]
Indeed, suppose that $\theta < \frac{1}{r'}$ (which me may always asume since $\frac{1}{l \log(l+2)^{r'\theta}}$ is decreasing on $\theta$). Thus, there is some $C_{r',\theta}>0$ such that
\[
\left( \sum_{l = 1}^n\frac{1}{l \log(l+2)^{r'\theta}} \right)^{1/r'} \le C_{r',\theta} \left( \int_{l = 2}^n\frac{1}{x \log(x)^{r'\theta}} dx\right)^{1/r'}  = C_{r',\theta} \left( \int_{l = \log(2)}^{\log(n)} \frac{1}{y^{r'\theta}} dy \right)^{1/r'} \le C_{r',\theta} \log(n)^{-\theta + \frac{1}{r'}}, 
\]

Then, $\frac{1}{\log(n+1)^{1/r'}} \left( \sum_{l = 1}^n\frac{1}{l \log(l+2)^{r'\theta}} \right)^{1/r'} \leq C_{r',\theta} \log(n)^{-\theta} \to 0.$

On the other hand, $z \in B_{\ell_r}$ (note that $|z_{n}| \le |w_{n}|$ for every $n$ and $w \in B_{\ell_r}$), then
\[
2e \| \id: m_{\Psi_r} \to \ell_r \|^r \left(  \limsup_{n \to \infty} \frac{\sum_{k=1}^n z_l^*}{\log(n+1)^{1-1/r}}\right)^r  + \| z \|_{\ell_r}^r  = \| z \|_{\ell_r}^r < 1,
\]
and, by Theorem~\ref{teo: mon H_infty}, $z \in \mon H_\infty (B_{\ell_r})$. \\
We give now the proof of \eqref{reis}. Take $z= \left( \frac{1}{K n^{1/r'} }  w_n \right)_{n \ge 1}$ with $w \in B_{\ell_r}$, and note that $\| z \|_{\ell_r}^r < \frac{1}{K^r} $. Proceeding as before we get
\[
K \frac{1}{\log(n+1)^{1/r'}} \sum_{l = 1}^n  z_l^* 
     \le \frac{1}{\log(n+1)^{1/r'}} \left( \sum_{l = 1}^n\frac{1}{l } \right)^{1/r'} \le 1.
\]
Since $K = ( 2e \| \id: m_{\Psi_r} \to \ell_r \|^r + 1)^{1/r}$,
\[
    2e \| \id: m_{\Psi_r} \to \ell_r \|^r  \left(  \limsup_{n \to \infty} \frac{\sum_{k=1}^n z_l^*}{\log(n+1)^{1-1/r}}\right)^r  + \| z \|_{\ell_r}^r 
    < (2e \| \id: m_{\Psi_r} \to \ell_r \|^r + 1)\frac{1}{K^r}  = 1.
\]
Again Theorem~\ref{teo: mon H_infty} gives the conclusion.
\end{proof}

\section{Lower inclusions for the set of monomial convergence of \texorpdfstring{$\mathcal P(^m \ell_r)$}.}\label{mhomogeneous}

We turn now our attention to the set of monomial convergence of homogeneous polynomials. We fix $1 < r \leq 2$ and $m \geq 2$ and define $q=(mr')'=\frac{mr}{r(m-1)+1}$. As we already pointed out, we know from \cite[Theorem~5.1]{bayart2016monomial} and \cite[Example~4.6]{defant2009domains} that 
\[
\ell_{q- \varepsilon} \subset \mon \mathcal P(^m\ell_r) 
\subset \ell_{q,\infty}
\] 
for every $\varepsilon>0$. Our aim now is to tighten this lower bound. We find a lower inclusion that gets narrower when $m$ gets bigger. 

\begin{theorem}\label{teo: main poli}
Fix $1 < r \le 2$ and, for each $m \geq 2$, define $q := (mr')'$. Then $\displaystyle \ell_{q} \subset\mon \Pp(^2 \ell_r)$; $\displaystyle \ell_{q,2} \subset\mon \Pp(^3 \ell_r)$;  $\displaystyle \ell_{q,\frac{3+\sqrt{5}}{2}} \subset\mon \Pp(^4 \ell_r)$
and
\[
\ell_{q,\frac{m}{\log(m)}} \subset\mon \Pp(^m \ell_r).
\]
for $m \ge 5$.
\end{theorem}

We start with  Theorem~\ref{teo ellq}, which proves the case $m=2$ in the previous theorem and also provides an elementary proof of the fact that $\ell_q$ is contained in $\mon \Pp(^m \ell_r)$. We even get a very good estimate for the sums. We will show later in Remark~\ref{rem: hyper} (see also the comments after it) that for $m\ge 3$ something more can be  achieved.
We need first a lemma.

\begin{lemma}\label{lem: alpha(j)}
	Let $r>1$. There exists $C_{r}>0$ such that, for every $m$,
	\[
	\sup \Big\{ \frac{m^{m/r}}{m!} \frac{n_{1}!}{n_{1}^{n_{1}/r}}  \cdots \frac{n_{k}!}{n_{k}^{n_{k}/r}} \colon k \in \mathbb{N}, n_{1}, \ldots , n_{k} \in \mathbb{N}\setminus \{0\} ,  n_{1} + \cdots + n_{k} =m  \Big\}
	\leq C_{r} m^{\frac{e^{\frac{1}{r-1}} - 1}{2}} \,.
	\]
\end{lemma}	
\begin{proof}
We proceed by induction on $m$. The statement is trivially satisfied for $m=2$ and  we assume it holds for $m-1$. Fix then $k$ and choose $n_{1}, \ldots , n_{k} \in \mathbb{N}$, all non-zero, such that $n_{1} + \cdots + n_{k} =m$. We may assume $n_{1} \geq \cdots \geq n_{k} \geq 1$. We consider two possible cases. First, if $k < e^{\frac{1}{r-1}}$ Stirling formula and the fact that $n_{j} \leq m$ for every $j$ yield
\begin{multline*}
\frac{m^{m/r}}{m!} \frac{n_{1}!}{n_{1}^{n_{1}/r}}  \cdots \frac{n_{k}!}{n_{k}^{n_{k}/r}}
\leq \frac{1}{\sqrt{2\pi m}} \frac{e^{m}}{m^{m/r'}} \prod_{j=1}^{k} \frac{\sqrt{2\pi n_{j}} n_{j}^{n_{j}/r'} e^{1/(12n_{j})}}{e^{n_{j}}} \\
\leq \big( 2\pi \big)^{\frac{k-1}{2}} e^{\sum_{j=1}^{k} \frac{1}{12n_{j}}} 
\bigg( \frac{n_{1}^{n_{1}} \cdots   n_{k}^{n_{k}}}{m^{m}} \bigg)^{\frac{1}{r'}} 
\bigg( \frac{n_{1} \cdots   n_{k}}{m} \bigg)^{\frac{1}{2}} 
\leq \big( 2\pi \big)^{\frac{k-1}{2}} e^{\sum_{j=1}^{k} \frac{1}{12j}} m^{\frac{k-1}{2}} \\
\leq \big( 2\pi \big)^{\frac{e^{\frac{1}{r-1}} - 1}{2}} e^{\frac{r}{12(r-1)}} m^{\frac{e^{\frac{1}{r-1}}-1}{2}} \,.
\end{multline*}
On the other hand, if $k \geq e^{\frac{1}{r-1}}$ we have
\begin{equation} \label{marianne}
\frac{m^{m/r}}{m!} \frac{n_{1}!}{n_{1}^{n_{1}/r}}  \cdots \frac{n_{k}!}{n_{k}^{n_{k}/r}}
= \Big(\frac{m}{m-1} \Big)^{\frac{m-1}{r}} \frac{1}{m^{1/r'}} \frac{(m-1)^{(m-1)/r}}{(m-1)!} \frac{n_{1}!}{n_{1}^{n_{1}/r}}  \cdots \frac{n_{k-1}!}{n_{k-1}^{n_{k-1}/r}} \frac{n_{k}!}{n_{k}^{n_{k}/r}} \,.
\end{equation}
If $n_{k}=1$ then $n_{1} + \cdots + n_{k-1} =m-1$ and we may use the induction hypothesis and the fact that $k \leq m$ to have
\[
\frac{m^{m/r}}{m!} \frac{n_{1}!}{n_{1}^{n_{1}/r}}  \cdots \frac{n_{k}!}{n_{k}^{n_{k}/r}}
\leq \Big(\frac{m}{m-1} \Big)^{\frac{m-1}{r}} \frac{1}{k^{1/r'}} C_{r} (m-1)^{\frac{e^{\frac{1}{r-1}} - 1}{2}}
\leq C_{r} e^{1/r}  \frac{1}{e^{\frac{1}{(r-1)r'}}}(m-1)^{\frac{e^{\frac{1}{r-1}} - 1}{2}}
\leq  C_{r} m^{\frac{e^{\frac{1}{r-1}} - 1}{2}} \,.
\]
Finally, if $n_{k}>1$ then
\[
\frac{(n_{k}-1)^{\frac{n_{k}-1}{r'}} n_{k}}{n_{k}^{n_{k}/r}}
= \Big( \frac{n_{k}-1}{n_{k}} \Big)^{\frac{n_{k}-1}{r'}} n_{k}^{\frac{1}{r'}}
\leq n_{k}^{\frac{1}{r'}} \,.
\]
We may use again the  induction hypothesis and the fact that $n_{k} \leq m/k$ to obtain from \eqref{marianne}
\[
\frac{m^{m/r}}{m!} \frac{n_{1}!}{n_{1}^{n_{1}/r}}  \cdots \frac{n_{k}!}{n_{k}^{n_{k}/r}}
\leq \Big(\frac{m}{m-1} \Big)^{\frac{m-1}{r}} \Big( \frac{n_{k}}{m} \Big)^{1/r'} C_{r}(m-1)^{\frac{e^{\frac{1}{r-1}} - 1}{2}}
\leq \Big(\frac{m}{m-1} \Big)^{\frac{m-1}{r}} \frac{1}{k^{1/r'}} C_{r} (m-1)^{\frac{e^{\frac{1}{r-1}} - 1}{2}} \,.
\]
From here we conclude as in the previous case.
\end{proof}

\begin{theorem}\label{teo ellq}
For each $1 < r \le 2$, there exists  $d_{r}>1$ such that for each $m$ and $n$, every $P \in \mathcal P(^m \mathbb C^n)$ and all $z \in \mathbb{C}^{n}$
\begin{equation}\label{hyper}
     \sum_{1 \leq j_1 \leq \dots \leq j_{m} \leq n}  \vert c_{\mathbf j}(P) z_{j_1} \dots  z_{j_m} \vert 
     \le m^{d_{r}} \Vert P \Vert_{\mathcal P(^m\ell_r^n)} \Vert z \Vert_{\ell_q^n}^{m} \,,
\end{equation} 
where $q := (mr')'$. In particular
\[
\ell_{q} \subset\mon \Pp(^m \ell_r).
\]
\end{theorem}

\begin{proof}
Clearly it is enough to show  \eqref{hyper} and, by \eqref{basta la desi para z*} (see also \cite[Lemma~10.15]{defant2019libro}), we may assume without loss of generality $z=z^*$. First of all, by
H\"older inequality we have
\begin{multline*}
\sum_{1 \leq j_1 \leq \dots \leq j_{m} \leq n}  \vert c_{\mathbf j}(P) z_{j_1} \dots z_{j_{m-1}} z_{j_m} \vert = \sum_{1 \leq j_1 \leq \dots \leq j_{m-1} \leq n} \vert  z_{j_1} \dots z_{j_{m-1}} \vert \sum_{j_m=j_{m-1}}^n \vert c_{\mathbf j}(P) z_{j_m}\vert  \\
 \leq \sum_{1 \leq j_1 \leq \dots \leq j_{m-1} \leq n} \vert  z_{j_1} \dots z_{j_{m-1}} \vert \bigg(\sum_{j_m=j_{m-1}}^n  \vert c_{\mathbf j}(P) \vert^{r'}\bigg)^{\frac{1}{r'}} \bigg(\sum_{j_m=j_{m-1}}^n \vert z_{j_m}^r \vert \bigg)^{\frac{1}{r}} 
\end{multline*}
Using Lemma~\ref{lem: BDS} together with the fact that for every $(\mathbf i,k)\in \mathcal J(m-1,n)$ we have $\big(\frac{(m-1)^{m-1}}{\alpha(\mathbf i,k)^{\alpha(\mathbf i,k)}}\big)\le e(m-1)\big(\frac{(m-2)^{m-2}}{\alpha(\mathbf i)^{\alpha(\mathbf i)}}\big)$ we obtain
\begin{multline*}
 \sum_{\mathbf j\in\mathcal J(m,n)}  \vert c_{\mathbf j}(P) z_{\mathbf j} \vert  \\
 \le e^{1+\frac{1}{r}}(m-1)^{\frac{1}{r}}m \Vert P \Vert_{\mathcal P(^m\ell_r^n)} \sum_{ j_{m-1}=1}^n |z_{j_{m-1}}| \sum_{\mathbf i\in\mathcal J(m-2,j_{m-1})} \vert  z_{\mathbf i}\vert  \bigg(\frac{(m-2)^{m-2}}{\alpha(\mathbf i)^{\alpha(\mathbf i)}}\bigg)^{\frac{1}{r}} \bigg(\sum_{j_m=j_{m-1}}^n \vert z_{j_m} \vert^r \bigg)^{\frac{1}{r}} 
\end{multline*}
For each fixed $1 \leq k \leq n$ we have, using \eqref{strauss}  and Lemma~\ref{lem: alpha(j)} (we write $a_{r} = \frac{e^{\frac{1}{r-1}} - 1}{2}$) and the fact that $q \leq r$
\begin{multline*}
|z_{k}| \sum_{\mathbf i\in\mathcal J(m-2,k)} \vert  z_{\mathbf i}\vert \bigg(\frac{(m-2)^{m-2}}{\alpha(\mathbf i)^{\alpha(\mathbf i)}}\bigg)^{\frac{1}{r}} \bigg(\sum_{j=k}^n \vert z_{j} \vert^r \bigg)^{\frac{1}{r}}
\leq |z_{k}| \sum_{\mathbf i\in\mathcal J(m-2,k)} \vert  z_{\mathbf i}\vert |\mathbf i|  \frac{(m-2)^{(m-2)/r}}{\alpha(\mathbf i)^{\alpha(\mathbf i)/r}|\mathbf i|} \Big( \vert z_{k} \vert^{r-q} \sum_{j=k}^n |z_{j}|^q  \Big)^{\frac{1}{r}} \\
=  C_{r}(m-2)^{a_{r}}  |z_{k}|^{2-\frac{q}{r}} \sum_{i_{1}, \ldots, i_{m-2} =1}^{k} \vert z_{i_{1}} \cdots z_{i_{m-2}} \vert \Big( \sum_{j=k}^n |z_{j}|^q  \Big)^{\frac{1}{r}} 
=  C_{r}(m-2)^{a_{r}} \Vert z \Vert_{\ell_{q}}^{\frac{q}{r}}  |z_{k}|^{2-\frac{q}{r}} \Big( \sum_{i=1}^{k} \vert z_{i}  \vert \Big)^{m-2} \\
\leq C_{r}(m-2)^{a_{r}} \Vert z \Vert_{\ell_{q}^{n}}^{\frac{q}{r} + m-2}  |z_{k}|^{2-\frac{q}{r}} k^{\frac{m-2}{q'}} \,.
\end{multline*}
Now
\[
\sum_{k=1}^{n} |z_{k}|^{2-\frac{q}{r}} k^{\frac{m-2}{q'}}
=\Vert z \Vert_{\ell_{q,2-q/r}^n}^{2-\frac{q}{r}}
\leq \Vert z \Vert_{\ell_{q}^n}^{2-\frac{q}{r}}
\]
because $2-\frac{q}{r}\ge q$ for $m\ge 2$. This altogether gives
\[
\sum_{\mathbf j\in\mathcal J(m,n)}  \vert c_{\mathbf j}(P) z_{\mathbf j} \vert  
\leq K_{r}  m (m-1)^{\frac{1}{r}} (m-2)^{a_{r}}  \Vert P \Vert_{\mathcal P(^m\ell_r^n)} \Vert z \Vert_{\ell_{q}^n}^{m} \,. \qedhere
\]
\end{proof}

This gives the case $m=2$ in Theorem~\ref{teo: main poli}. We face now the problem of getting the result for other $m$'s.
The general philosophy is always  to try to get a bound as that in \eqref{hyper}, where in the right-hand-side we have some constants that depend on $r$ and $m$ (but not on $n$, the number of variables), the norm of the polynomial and the norm of $z$ in some space $X$. This then implies $X \subset \mon \mathcal{P}(^{m} \ell_{r})$. What we do is to take the sum as depending on $m$ different variables; that is, for each polynomial $P$ we consider
\begin{equation} \label{goldberg}
  \sum_{1 \leq j_1 \leq \cdots \leq j_{m} \leq n}  \vert c_{\mathbf j}(P) z_{j_1}^{(1)} \dots  z_{j_m}^{(m)} \vert 
\end{equation}
with $z^{(1)} , \ldots , z^{(m)} \in \mathbb{C}^{n}$ and then try to get an estimate that involves 
the norms of the $z^{(j)}$ in (possibly) different spaces. This then gives that the smallest of these spaces is contained in the set of monomial convergence (see Remark~\ref{manel}). We do this
(giving the proof of Theorem~\ref{teo: main poli}) in two stages (that we present in the following two subsections). First we give an estimate for the sum that involves both $\ell_{q,1}$ and 
$\ell_{q,\infty}$ norms (the precise statement is given in Proposition~\ref{teoremamulti}). Then we interpret this inequality as operators from $\ell_{q,\infty}\times\dots\times\ell_{q,\infty}\times \ell_{q,1}\times\ell_{q,\infty}\times\dots\times \ell_{q,\infty}$ to $\ell_1(\mathcal J (m,n))$
and use interpolation techniques to improve the $\ell_{q,1}$-norm (by weakening the $\ell_{q, \infty}$-norm). This is done in Theorem~\ref{teoremamulti 2}. What happens here is that, since in
the estimate in Proposition~\ref{teoremamulti} some of the variables have to be decreasing, we cannot use general multilinear interpolation, but interpolation in cones (a more detailed explanation is given in Section~\ref{sec:interp conos}).

\subsection{First bound for the sum}
As we announced, our first step towards the proof of Theorem~\ref{teo: main poli} is to get a bound for a sum like that in \eqref{goldberg}. This becomes the main result of this section.

\begin{proposition}\label{teoremamulti}
Let $1 < r \leq 2$ and $m \geq 2$. Define  $q:=(mr')'$. There exists $C_{m,r} >1$ so that for every $n \in \mathbb{N}$, every $P \in \mathcal{P}(^{m} \mathbb{C}^{n})$, every $z^{(1)}, \ldots , z^{(m)} \in \mathbb{C}^{n}$ and $1 \leq k \leq m-1$ we have
\[
 \sum_{1 \leq j_1 \leq \dots \leq j_{m} \leq n} \big\vert c_{\mathbf j}(P) z_{j_1}^{(1)} \cdots  z_{j_{k}}^{(k)} z_{j_{k+1}}^{(k+1)*} \cdots  z_{j_m}^{(m)*} \big\vert 
\leq C_{m,r} \Vert z^{(k)} \Vert_{\ell_{q,1}} \prod_{i\neq k} \Vert z^{(i)} \Vert_{\ell_{q , \infty}} \|P\|_{\mathcal P(^m\ell_r^n)} \,.
\]
\end{proposition}

The proof requires some work, that we prepare with a few lemmas. But before let us make a couple of elementary comments. First of all, by definition, 
\begin{equation} \label{drum}
z_{k}^{*} \leq \Vert z \Vert_{\ell_{q, \infty}} \frac{1}{k^{1/q}}
\end{equation} 
for every $z \in \mathbb{C}^{n}$ and, then
\begin{equation} \label{blakey}
\sum_{k=N}^{M} z_{k}^{*} \leq \Vert z \Vert_{\ell_{q, \infty}} \sum_{k=N}^{M}  \frac{1}{k^{1/q}} \,.
\end{equation}
Also, for $1 \neq \alpha < 0$,
\begin{equation} \label{messengers}
\sum_{k=N}^{M} n^{\alpha} = N^{\alpha} + \sum_{k=N+1}^{M} n^{\alpha}
\leq  N^{\alpha} + \int_{N}^{M} x^{\alpha} dx 
=  N^{\alpha} +\frac{1}{\alpha + 1} \big( M^{\alpha + 1} - N^{\alpha + 1} \big) \,.
\end{equation}

\begin{lemma} \label{primeros}
Let $n, k \geq 1$ and $1 \leq q < \infty$. Then for every $z^{(1)}, \ldots, z^{(k)} \in \mathbb{C}^{n}$ and $1 \leq j \leq n$ we have
\[
\sum_{1 \leq j_1 \leq \dots \leq j_{k} \leq j} \vert z^{(1)}_{j_1} \dots z^{(k)}_{j_{k}} \vert 
\leq (q')^{k} j^{\frac{k}{q'}} \prod_{1 \leq i \leq k} \Vert z^{(i)} \Vert_{\ell_{q,\infty}}.
\]
\end{lemma}
\begin{proof}
We proceed by induction on $k$. For $k=1$ the statement is a straightforward consequence of \eqref{blakey} and \eqref{messengers}. Assume that the result holds for $k-1$. Then
\begin{multline*}
 \sum_{1 \leq j_1 \leq \dots \leq j_{k} \leq j} \vert z^{(1)}_{j_1} \cdots z^{(k)}_{j_{l}} \vert   
= \sum_{ j_k =1}^{j} \vert z^{(k)}_{j_k}\vert   \,
\Big( \sum_{1 \leq j_1 \leq \dots \leq j_{k-1} \leq j_{k}} \vert z^{(1)}_{j_1} \dots z^{(k-1)}_{j_{k-1}} \vert \Big) \\
\leq (q')^{k-1}  \prod_{1 \leq i \leq k-1} \Vert z^{(i)} \Vert_{\ell_{q,\infty}}
j_k^{\frac{k-1}{q'}} \sum_{ j_k =1}^{j} \vert z^{(k)}_{j_k}\vert 
 \le (q')^{k} j^{\frac{k-1}{q'}} j^{\frac{1}{q'}} 
\prod_{1 \leq i \le k} \Vert z^{(i)} \Vert_{\ell_{q,\infty}}    \,,
\end{multline*}
 which concludes the proof.
\end{proof}

\begin{lemma}\label{segundos}
Let $1 < r \leq 2$, $m\ge 3$ and $n \in \mathbb{N}$. Fix $q:=(mr')'$ and  $1 \leq k \leq m-2$. For every $z^{(i_1)},\ldots,z^{(i_k)} \in \mathbb{C}^{n}$ and $1 \leq t \leq n$ we have
\[
\sum_{t \leq j_{1} \leq \dots \leq j_{k} \leq n} \vert z^{(i_1)*}_{j_{1}} \dots z^{(i_k)*}_{j_{k}} \vert j_{k}^{\frac{1}{r} - \frac{1}{q}} 
\leq \Big( \prod_{1 \le l \leq k} \big(\frac{mr'}{m-l-1}+\frac1{t} \big) \Big) 
t^{\frac{k+1}{q'}-\frac{1}{r'}} \Big( \prod_{1 \le l \leq k} \Vert z^{(i_l)} \Vert_{\ell_{q,\infty}} \Big).
\]
\end{lemma}
\begin{proof}
First of all let us note that a simple computation shows that $\frac{s}{q'} - \frac{1}{r'} \leq - \frac{1}{mr'} < 0$ for every $1 \leq s \leq  m-1$. We now proceed by induction on $k$. For $k=1$
we use \eqref{blakey} and \eqref{messengers} to have
\[
\sum_{j=t}^{n} \vert z_{j}^{*} \vert j^{\frac{1}{r} - \frac{1}{q}}
\leq \Vert z \Vert_{\ell_{q, \infty}} \sum_{j=t}^{n} j^{\frac{2}{q'} - \frac{1}{r'} -1}
\leq  \| z \|_{\ell_{q,\infty}} 
\Big(t^{\frac{2}{q'}-\frac{1}{r'}} - (\frac{2}{q'} - \frac{1}{r'})^{-1} 
t^{\frac{2}{q'}-\frac{1}{r'} + 1}\Big)
= \big( \frac{r'm}{m-2}+\frac1{t} \big) t^{\frac{2}{q'} - \frac{1}{r'}} \| z \|_{\ell_{q,\infty}}.
\]
Let us suppose now that the statement holds for $k-1$ and prove it for $k$. 
\begin{align*}
 \sum_{t \leq j_{1} \leq \dots \leq j_{k} \leq n}  & \vert z^{(i_1)*}_{j_{1}}  \cdots z^{(i_k)*}_{j_{k}} \vert j_{k}^{\frac{1}{r} - \frac{1}{q}} \\ 
	& = \sum_{j_{1}=t}^n \vert z^{(i_1)*}_{j_{1}} \vert \sum_{j_{1} \leq j_{2} \leq \dots \leq j_{k} \leq n} \vert z^{(i_2)*}_{j_{2}} \dots z^{(i_k)*}_{j_{k}} \vert j_{k}^{\frac{1}{r} - \frac{1}{q}} \\
& \leq \sum_{j_{1}=t}^n \vert z^{(i_1)*}_{j_{1}} \vert 
 \Big( \prod_{1 \le l \leq k-1} \big(\frac{mr'}{m-l-1}+\frac1{j_{1}} \big) \Big) 
j_{1}^{\frac{k}{q'}-\frac{1}{r'}} \Big( \prod_{2 \le l \leq k} \Vert z^{(i_l)} \Vert_{\ell_{q,\infty}} \Big) \\
& \leq  \Big( \prod_{1 \le l \leq k-1} \big(\frac{mr'}{m-l-1}+\frac1{t} \big) \Big)  \Big( \prod_{2 \le l \leq k} \Vert z^{(i_l)} \Vert_{\ell_{q,\infty}} \Big)
 \sum_{j_{1}=t}^n \vert z^{(i_1)*}_{j_{1}} \vert j_{1}^{\frac{k}{q'}-\frac{1}{r'}}  \\
& \leq  \Big( \prod_{1 \le l \leq k-1} \big(\frac{mr'}{m-l-1}+\frac1{t} \big) \Big)  \Big( \prod_{1 \le l \leq k} \Vert z^{(i_l)} \Vert_{\ell_{q,\infty}} \Big)
 \sum_{j_{1}=t}^n  j_{1}^{\frac{k+1}{q'}-\frac{1}{r'}-1}  \\
& \leq  \Big( \prod_{1 \le l \leq k-1} \big(\frac{mr'}{m-l-1}+\frac1{t} \big) \Big)  \Big( \prod_{1 \le l \leq k} \Vert z^{(i_l)} \Vert_{\ell_{q,\infty}} \Big)
t^{\frac{k+1}{q'} - \frac{1}{r'}}
\Big( \frac{1}{t} - \big( \frac{k+1}{q'} - \frac{1}{r'} \big)^{-1}  \Big) \\
& =  \Big( \prod_{1 \le l \leq k-1} \big(\frac{mr'}{m-l-1}+\frac1{t} \big) \Big)  \Big( \prod_{1 \le l \leq k} \Vert z^{(i_l)} \Vert_{\ell_{q,\infty}} \Big)
t^{\frac{k+1}{q'} - \frac{1}{r'}}
\Big( \frac{1}{t} + \frac{mr'}{m-k-1}  \Big) \,. \qedhere
\end{align*}
\end{proof}

For the following next we need the following well known Hardy-Littlewood rearrangement inequality (see for example \cite[Section~10.2, Theorem~368]{hardy1952inequalities}).

\begin{lemma} \label{reordenamiento}
Let $(a_k)_{k\in \mathbb N}$ and $(b_k)_{k\in \mathbb N}$ two non-increasing sequences of non-negative real numbers. Then, for every $ m\in \mathbb N$ and every injection $\sigma: \mathbb N \to \mathbb N$ we have
$$ \sum_{k=1}^m a_{\sigma(k)}  b_{k} \leq  \sum_{k=1}^m a_{k}  b_{k}.$$
\end{lemma}

\begin{lemma}\label{desp-acot-cj}
Let $1 < r \leq 2$, $m \geq 3$. Fix $q:=(mr')'$ and  $1 \leq k \leq m-2$. For every $z^{(1)},\ldots,z^{(k)} \in \mathbb{C}^{n}$ we have
\[
\sum_{1 \leq j_1 \leq \dots \leq j_{m-1} \leq n} \vert z^{(1)}_{j_1} \cdots z^{(k)}_{j_{k}}  z^{(k+1)*}_{j_{k+1}} \cdots z^{(m-1)*}_{j_{m-1}} \vert j_{m-1}^{\frac{1}{r} - \frac{1}{q}}  \leq (q'+1)^{m-2}
\Vert z^{(k)} \Vert_{\ell_{q,1}} \prod_{\substack{1 \leq i \leq m-1 \\ i  \neq k}} \Vert z^{(i)} \Vert_{\ell_{q,\infty}}\,.
\]
\end{lemma}
\begin{proof}
We begin by splitting the sum in a convenient way
\begin{multline*}
\sum_{1 \leq j_1 \leq \dots \leq j_{m-1} \leq n} \vert z^{(1)}_{j_1} \cdots z^{(k)}_{j_{k}}  z^{(k+1)*}_{j_{k+1}} \cdots z^{(m-1)*}_{j_{m-1}} \vert j_{m-1}^{\frac{1}{r} - \frac{1}{q}} \\
= \sum_{j_k=1}^n \vert z^{(k)}_{j_k}\vert 
\Big(  \sum_{j_k \leq j_{k+1} \leq \dots \leq j_{m-1} \leq n} \vert z^{(k+1)*}_{j_{k+1}} \dots z^{(m-1)*}_{j_{m-1}} \vert j_{m-1}^{\frac{1}{r} - \frac{1}{q}} \Big) 
\Big(  \sum_{1 \leq j_1 \leq \dots \leq j_{k-1} \leq j_k} \vert z^{(1)}_{j_1} \dots z^{(k-1)}_{j_{k-1}} \vert \Big) \,.
\end{multline*}
We fix $j_{k} $ and bound the first block using Lemma~\ref{segundos}, taking into account that we have now $m-k-1$ $z$'s and that $\frac{1}{j_{k}} + \frac{mr'}{m-l-1} \leq q'+1$ for every  $1 \leq l \leq m-k-1$,
\begin{multline*}
\sum_{j_k \leq j_{k+1} \leq \dots \leq j_{m-1} \leq n} \vert z^{(k+1)*}_{j_{k+1}} \dots z^{(m-1)*}_{j_{m-1}} \vert j_{m-1}^{\frac{1}{r} - \frac{1}{q}} \\
\leq j_{k}^{\frac{m-k}{q'} - \frac{1}{r'}} \Big(\prod_{1 \leq l \leq m-k-1}\frac{1}{j_{k}} + \frac{mr'}{m-l-1} \Big)
\Big( \prod_{k+1 \leq i \leq m-1} \Vert z^{(i)} \Vert_{\ell_{q,\infty}}  \Big) \\
\leq  j_{k}^{\frac{m-k}{q'} - \frac{1}{r'}} (q'+1)^{m-k-1} \prod_{k+1 \leq i \leq m-1} \Vert z^{(i)} \Vert_{\ell_{q,\infty}}  \,.
\end{multline*}
With this, and bounding the second block using Lemma~\ref{primeros} we get
\[
\sum_{1 \leq j_1 \leq \dots \leq j_{m-1} \leq n} \vert z^{(1)}_{j_1} \cdots z^{(k)}_{j_{k}}  z^{(k+1)*}_{j_{k+1}} \cdots z^{(m-1)*}_{j_{m-1}} \vert j_{m-1}^{\frac{1}{r} - \frac{1}{q}}
\leq (q'+1)^{m-2} \prod_{i \neq k} \Vert z^{(i)} \Vert_{\ell_{q,\infty}}  \sum_{j_k=1}^n \vert z^{(k)}_{j_k}\vert j_{k}^{\frac{k-1}{q'}+\frac{m-k}{q'}-\frac{1}{r'}} .
\]
It easy to see that $\frac{k-1}{q'}+\frac{m-k}{q'}-\frac{1}{r'}=\frac{1}{q}-1$. Therefore, using Lemma~\ref{reordenamiento} we have 
\[
\sum_{j_k=1}^n \vert z^{(k)}_{j_k} \vert j_k^{\frac{1}{q}-1} \leq \sum_{j_k=1}^n \vert (z^{(k)})^*_{j_k} \vert j_k^{\frac{1}{q}-1} = \Vert z^{(k)} \Vert_{\ell_{q,1}}. \qedhere
\]
\end{proof}

As it was the case for the study of holomorphic functions, Lemma~\ref{lem: BDS} (in fact \eqref{eq: BDS},  which is \cite[Lemma~3.5]{bayart2016monomial}) is a crucial tool for the proof of Proposition~\ref{teoremamulti}.

\begin{proof}[Proof of Proposition~\ref{teoremamulti}]
We begin by using H\"older's inequality and \eqref{eq: BDS} (noting that $\vert \mathbf{i} \vert \leq (m-1)!$ for every $\mathbf{i} \in \mathcal{J}(m-1,n)$) and \eqref{drum} to have
\begin{align*}
\sum_{1 \leq j_1 \leq \dots \leq j_{m} \leq n} & \big\vert c_{\mathbf j}(P) z_{j_1}^{(1)} \cdots  z_{j_{k}}^{(k)} z_{j_{k+1}}^{(k+1)*} \cdots  z_{j_m}^{(m)*} \big\vert 
= \sum_{1 \leq j_1 \leq \dots \leq j_{m-1} \leq n} \vert  z^{(1)}_{j_1} \dots z^{(m-1)*}_{j_{m-1}} \vert \sum_{j_m=j_{m-1}}^n \vert c_{\mathbf j}(P) z^{(m)*}_{j_m}\vert \\
& \leq \sum_{1 \leq j_1 \leq \dots \leq j_{m-1} \leq n} \vert  z^{(1)}_{j_1} \dots z^{(m-1)*}_{j_{m-1}} \vert 
\Big(\sum_{j_m=j_{m-1}}^n   c_{\mathbf j}(P)^{r'}\Big)^{\frac{1}{r'}} \Big(\sum_{j_m=j_{m-1}}^n \vert z^{(m)*}_{j_m} \vert^r \Big)^{\frac{1}{r}} \\
& \leq (m-1)!^{\frac{1}{r}}me^{1+\frac{m-1}{r}} \Vert P \Vert_{\mathcal P(^m\ell_r^n)} \Vert z^{(m)} \Vert_{\ell_{q,\infty}}
\sum_{1 \leq j_1 \leq \dots \leq j_{m-1} \leq n} \vert  z^{(1)}_{j_1} \dots z^{(m-1)*}_{j_{m-1}} \vert
\Big(\sum_{j_m=j_{m-1}}^n j_{m}^{-\frac{r}{q}}  \Big)^{\frac{1}{r}} \,.
\end{align*}
Observe now that, for each $N \in \mathbb{N}$ we have $N^{-r/q} \leq 2^{r/q} x^{-r/q}$ for every $N \leq x < N+1$. Then
\[
\sum_{j_m=j_{m-1}}^n j_{m}^{-\frac{r}{q}} \leq 2^{\frac{r}{q}} \int_{j_{m-1}}^{n}  x^{-\frac{r}{q}} dx
\leq 2^{\frac{r}{q}} \frac{q}{r-q} j_{m-1}^{1-\frac{r}{q}}\,.
\]
The proof now finishes with a straightforward application of Lemma~\ref{desp-acot-cj}.
\end{proof}

\subsection{Real interpolation on cones.}

\label{sec:interp conos}

What we are going to do now is to look at the inequalities for sums like in \eqref{strauss} from the point of view of multilinear mappings. We fix a polynomial $P \in \mathcal{P}(^{m} \mathbb{C}^{n})$ and consider the mapping $\mathbb{C}^{n} \times \cdots \times \mathbb{C}^{n} \to \ell_1(\mathcal J (m,n))$,
given by
\begin{equation} \label{niza}
(z^{(1)}, \dots, z^{(m)}) \mapsto \big( c_{\mathbf j}(P) z^{(1)}_{j_1} \dots z^{(m)}_{ j_{m}} \big)_{\mathbf j \in \mathcal J(m,n)}\,.
\end{equation}
Note that, since everything here is finite dimensional, the mapping is well defined. The idea is, then, to consider norms on the domain spaces so that the norm of this mapping is bounded by
a term involving the norm of the polynomial and some constant independent of $n$. Since the inequality that we get in Proposition~\ref{teoremamulti} requires some variables to be decreasing 
we have to restrict ourselves to cones of decreasing sequences. To be more precise, if we denote $\ell_{q, s}^{d} := \{ z \in \ell_{q,s} \colon \vert z \vert = z^{*}  \}$ for $1 \leq s \leq \infty$, Proposition~\ref{teoremamulti} tells us that there is a constant $C_{m,r}>1$ (independent of $P$ and $n$) such that, for every $1 \leq k \leq m-1$, the  mapping 
\begin{equation}\label{1079}
T_k: \underbrace{\ell_{q,\infty}^{n} \times \dots \times \ell_{q,\infty}^{n}}_{k-1} \times \ell_{q,1}^{n} \times \underbrace{(\ell_{q,\infty}^{n})^d \times \dots \times  (\ell_{q,\infty}^{n})^d}_{m-k} \to \ell_1(\mathcal J (m,n)),
\end{equation}
given by \eqref{niza} satisfies
\begin{equation} \label{voltaraser}
\Vert T_{k} \Vert \leq C_{m,r} \|P\|_{\mathcal P(^m\ell_r^n)}.
\end{equation}
All these mappings have the same defining formula (which is $m$-linear), so it is tempting to apply multilinear interpolation. But, since we need to restrict ourselves to the cone of non-increasing sequences in the last $m-k$ variables, we are not able to directly apply the classical multilinear interpolation results, but interpolation in cones. \\
For the general theory of interpolation we follow (and refer the reader to) \cite{BerLof76}. Since (as we have already explained) we have to consider linear operators on cones, we use the $K$-method of 
interpolation for operators on the cone of non-increasing sequences, as presented in \cite{cerda1996interpolation}. Then the main result of this section, from which Theorem~\ref{teo: main poli}, follows is the following.

\begin{theorem}\label{teoremamulti 2}
Let $1 < r \leq 2$ and $m \geq 3$. Define $q:=(mr')'$ and
\[
s = \begin{cases}
2 & \text{ if } m=3 \\
\frac{3+\sqrt{5}}{2}& \text{ if } m=4 \\
\frac{m}{\log(m)}& \text{ if } m \geq 5
\end{cases}
\]
There exists a constant $C_{m,r}\geq 1$ such that, for every $P \in \mathcal P(^m \mathbb{C}^n)$ the $m$-linear mapping
\[
T: \underbrace{(\ell_{q,s}^{n})^d \times \dots \times (\ell_{q,s}^{n})^d}_{m-1}\times (\ell_{q,\infty}^{n})^d \to \ell_1(\mathcal J (m,n))
\]
given by
\[
(z^{(1)}, \dots, z^{(m)}) \mapsto \big( c_{\mathbf j}(P) z^{(1)}_{j_1} \dots z^{(m)}_{ j_{m}} \big)_{\mathbf j \in \mathcal J(m,n)}
\]
satisfies
\[
\Vert T \Vert \leq C_{m,r} \|P\|_{\mathcal P(^m\ell_r^n)}\,.
\]
\end{theorem}

\begin{remark} \label{manel}
If we take $z^{(1)} = \ldots = z^{(m)} = z$ and observe that $\Vert z \Vert_{\ell_{q,\infty}} \leq \Vert z \Vert_{\ell_{q,s}}$, Theorem~\ref{teoremamulti 2} gives
\[
  \sum_{1 \leq j_1 \leq \cdots \leq j_{m} \leq n}  \vert c_{\mathbf j}(P) z_{j_1}^{*} \cdots  z_{j_m}^{*} \vert \leq C_{m,r} \Vert z \Vert_{\ell_{q,s}}^{m} \|P\|_{\mathcal P(^m\ell_r^n)}
\]
for every $P \in \mathcal P(^m \mathbb{C}^n)$ and $z \in \mathbb{C}^{n}$. A standard argument shows that $z^{*} \in \mon \mathcal{P}(^{m}\ell_{r})$ for every $z \in \ell_{q,s}$ and,
then, Corollary~\ref{cor: familias de reordenamiento} implies $\ell_{q,s} \subset \mon  \mathcal{P}(^{m}\ell_{r})$. This gives Theorem~\ref{teo: main poli}.
\end{remark}

Before we proceed, let us fix some notation. Given a Banach function lattice $X$ (in particular a sequence space or a finite dimensional Banach space, on which we are mainly interested), we write $X^d$ for the cone of non-increasing functions in $X$. If $Y$ is any Banach space and $S: X \to Y$ is a linear operator we can restrict it to the cone and denote
\begin{equation} \label{copland}
\Vert S : X^{d} \to Y \Vert = \inf \{ \Vert S(x) \Vert_{Y} \colon x \in X^{d} , \, \Vert x \Vert< 1  \} \,.
\end{equation}
Clearly neither is $X^{d}$ a vector space, nor is $\Vert S \Vert$ a norm. We will later use an analogous notation for $m$-linear mappings. We are now ready to state our main tool to interpolate in cones. It  is a direct corollary of  \cite[Theorem~1--(b)]{cerda1996interpolation} (recall that we are using the notation as introduced there).
\begin{theorem}\label{int op conos}
Given a pair of quasi-Banach function lattices $(X_0,X_1)$, a pair of quasi-Banach spaces $(Y_0,Y_1)$ and a linear operator $S$ defined both $X_{0} \to Y_{0}$ and $X_{1} \to Y_{1}$ with
\[
 \| S : X_0^d \longrightarrow Y_0 \|  \le M_0 \quad \text{ and } \quad
 \| S : X_1^d \longrightarrow Y_1 \|  \le M_1 \,.
\]
Then for every $0 < \theta < 1$ the operator $ S : (X_0^d,X_1^d)_{\theta,a} \longrightarrow (Y_0,Y_1)_{\theta,a} $ is well defined and
\[
\| S : (X_0^d,X_1^d)_{\theta,a} \longrightarrow (Y_0,Y_1)_{\theta,a} \| \le M_0^{1-\theta} M_1^\theta.
\]
\end{theorem}

We are going to apply this to Lorentz sequence spaces. In this case, it was proved in  \cite{sagher1972application} (see also \cite[Theorem~4]{cerda1996interpolation}) that
\[
(\ell^d_{q,p_0},\ell^d_{q,p_1} )_{\theta,a} = (\ell_{q,p_0},\ell_{q,p_1} )_{\theta,a}^{d} \,.
\]
On the other hand, it is known (see for example \cite[Theorem~5.3.1]{BerLof76}) that  whenever $\frac{1}{p} = \frac{1-\theta}{p_0}+\frac{\theta}{p_1}$ we have
\[
(\ell_{q,p_0},\ell_{q,p_1} )_{\theta,p} = \ell_{q,p},
\]
and therefore
\begin{equation}\label{int lorentz conos}
(\ell^d_{q,p_0},\ell^d_{q,p_1} )_{\theta,p} = \ell^d_{q,p}.
\end{equation}
Finally \cite[Theorem~3.7.1]{BerLof76} gives that (if $p_{0}, p_{1}, p$ are related as before)
\begin{equation}\label{int lorentz dual}
(\ell_{q,p_0}',\ell_{q,p_1}' )_{\theta,p} = (\ell_{q,p_0},\ell_{q,p_1} )_{\theta,p}'   = \ell_{q,p}' \,.
\end{equation}

The idea now is to use Theorem~\ref{int op conos} to interpolate multilinear mappings. Let us explain how we are going to do this. Let $X_{1}, \ldots , X_{m}$ be Banach function lattices (in our case they will always be finite dimensional Lorentz spaces), $Y$ some Banach space ($\ell_{1}(\mathcal{J}(m,n)$ for us) and some continuous $m$-linear $T: X_{1} \times \cdots \times X_{m} \to Y$ (for us given by \eqref{niza}). Now we fix $1 \leq j \neq k \leq m$ and, for each $i \neq j,k$ pick $z^{(i)} \in X_{i}$ and $\varphi \in Y'$ and consider $v= (z^{(1)} , \ldots , z^{(m)}, \varphi)$. Now we define
\begin{equation} \label{carminho}
T_{v} : X_{j} \to X_{k}' \quad \text{ by } \quad \big(T_v(z^{(j)}) \big)(z^{(k)}) = \varphi (T (z^{(1)} , \ldots , z^{(m)})).
\end{equation}
An easy computation shows that
\begin{equation} \label{apalachian}
\Vert T_{v} \Vert \leq \Vert \varphi \Vert \, \Vert T \Vert \prod_{i \neq j,k} \Vert z^{(i)} \Vert  \,.
\end{equation}
Observe that in this procedure we may consider $X_{i}^{d}$ for every $i$ except for $i=k$, getting the same estimate for the norm (defining the ``norm'' for multilinear mappings on cones with the same idea as in \eqref{copland}). We are now ready to present the main technical tool for the proof of Theorem~\ref{teoremamulti 2}.

\begin{lemma}\label{induction lemma}
Let $m \ge 3$, $1 < r \leq 2$, define $q:=(mr')'$ and let $C_{m,r}$ be the constant from Proposition~\ref{teoremamulti}. For each  $0 < \theta < 1$, every $P \in \mathcal P(^m \mathbb{C}^n)$ and all $1 \le k \le m-2$ the $m$-linear mapping
\[
T^k(\theta): \big(\ell_{q, ( \frac{1}{1-\theta} )^{k}}^n \big)^d \times 
\underbrace{\big(\ell_{q,\frac{1}{\theta}}^n\big)^d \times \cdots \times \big(\ell_{q,\frac{1}{\theta}}^n\big)^d}_{k} \times 
\underbrace{\big(\ell_{q,\infty}^n\big)^d \times \cdots \times \big(\ell_{q,\infty}^n\big)^d}_{m-k-1}  
\to \ell_1(\mathcal J (m,n))
\]
given by \eqref{niza} satisfies
\[
\Vert T^{k} (\theta) \Vert \leq C_{m,r} \|P\|_{\mathcal P(^m\ell_r^n)}.
\]
\end{lemma}
\begin{proof}
We proceed by induction on $k$ and begin with the case $k=1$. We consider the mappings (see \eqref{1079})
\begin{gather*}
T_1:  \ell_{q,1}^{n} \times \underbrace{(\ell_{q,\infty}^{n})^d \times (\ell_{q,\infty}^{n})^d \times\dots \times  (\ell_{q,\infty}^{n})^d}_{m-1} \to \ell_1(\mathcal J (m,n)) \\
T_2: \ell_{q,\infty}^{n} \times \ell_{q,1}^{n} \times \underbrace{(\ell_{q,\infty}^{n})^d \times \dots \times  (\ell_{q,\infty}^{n})^d}_{m-2} \to \ell_1(\mathcal J (m,n)).
\end{gather*}
We fix $z^{(3)} , \ldots , z^{(m)} \in (\ell_{\infty}^n)^d$ and $\varphi \in \Big( \ell_1(\Jj (m,n)) \Big)'$ and writing $v = (z^{(3)} , \ldots , z^{(m)}, \varphi)$ define, following \eqref{carminho}, two linear operators
\[
(T_1)_v : \big(\ell_{q,\infty}^n \big)^d  \to \big( \ell_{q,1}^{n} \big)' \quad \text{ and } \quad
(T_2)_v : \big(\ell_{q,1}^n \big)^d  \to \big( \ell_{q,\infty}^{n} \big)' 
\]
that, by \eqref{voltaraser} and \eqref{apalachian}, satisfy (for $i=1,2$)
\[
\| (T_i)_v \| \le C_{m,r} \Vert P \Vert_{\mathcal{P}(^{m} \ell_{r}^{n})} \| z^{(3)} \|_{\ell_{q,\infty}}  \cdots \| z^{(m)} \|_{\ell_{q,\infty}} \| \varphi \|_{\ell_1(\Jj (m,n))'} \,.
\]
Now we interpolate, using Theorem~\ref{int op conos} and equations \eqref{int lorentz conos} and \eqref{int lorentz dual}, to have
\[
\Big\Vert \big( T^1(\theta) \big)_v : \big( \ell_{q,\frac{1}{\theta}}^{n} \big)^d \to \big( \ell_{q,\frac{1}{1-\theta}}^{n} \big)' \Big\Vert 
\le C_{m,r} \Vert P \Vert_{\mathcal{P}(^{m} \ell_{r}^{n})} \| z^{(3)} \|_{\ell_{q,\infty}}  \cdots \| z^{(m)} \|_{\ell_{q,\infty}} \| \varphi \|_{\ell_1(\Jj (m,n))'}
\]
for every $0 < \theta < 1$. This immediately gives (just taking supremums)
\[
\big\Vert T^{1}(\theta): \ell_{q,\frac{1}{1-\theta}}^{n} 
\times \big( \ell_{q,\frac{1}{\theta}}^{n} \big)^d \times  
\underbrace{\big(\ell_{q,\infty}^{n} \big)^d \times \dots \times \big( \ell_{q,\infty}^{n} \big)^d}_{m-2} \to \ell_1(\mathcal J (m,n)) \big\Vert 
\le C_{m,r} \Vert P \Vert_{\mathcal{P}(^{m} \ell_{r}^{n})} \,.
\]
Now let us assume that, for $1 \leq k \leq m-2$, 
\[
T^{k-1}(\theta): \ell_{q, \left( \frac{1}{1-\theta} \right)^{k-1}}^n \times \underbrace{(\ell_{q,\frac{1}{\theta}}^n)^d \times \dots \times (\ell_{q,\frac{1}{\theta}}^n)^d}_{k-1} \times \underbrace{(\ell_{q,\infty}^n)^d \times\dots \times (\ell_{q,\infty}^n)^d}_{m-k}  \to \ell_1(\mathcal J (m,n))
\]
has norm $\le C_{m,r} \Vert P \Vert_{\mathcal{P}(^{m} \ell_{r}^{n})} $. On the other hand consider the mapping defined by Theorem~\ref{teoremamulti} (see \eqref{1079})
\[
T_{k+1} :   \underbrace{\ell_{q,\infty}^n \times \dots \times \ell_{q,\infty}^n}_{k} \times \ell_{q,1}^n \times \underbrace{(\ell_{q,\infty}^n)^d \times \dots \times  (\ell_{q,\infty}^n)^d}_{m-k-1} \to \ell_1(\mathcal J (m,n))
\]
that (recall \eqref{apalachian}) also has norm $\le C_{m,r} \Vert P \Vert_{\mathcal{P}(^{m} \ell_{r}^{n})} $. Since $\Vert \ell_{q, \frac{1}{\theta}}^n \hookrightarrow \ell_{q,\infty}^n \Vert =1$ we have (recall \eqref{copland})
\[
T_{k+1} : \ell_{q,\infty}^n  \times \underbrace{(\ell_{q, \frac{1}{\theta} }^n)^d \times \dots \times (\ell_{q, \frac{1}{\theta} }^n)^d}_{k-1} \times \ell_{q,1} \times \underbrace{(\ell_{q,\infty}^n)^d \times \dots \times  (\ell_{q,\infty}^n)^d}_{m-k-1} \to \ell_1(\mathcal J (m,n))
\]
has again norm bounded by $C_{m,r} \Vert P \Vert_{\mathcal{P}(^{m} \ell_{r}^{n})} $. We fix $\varphi \in \big( \ell_1(\Jj (m,n)) \big)'$ and $z^{(i)} \in (\mathbb{C}^n)^d$ for $i \neq 1,k$ and, taking $v=( z^{(2)}, \ldots, z^{(k)}, z^{(k+2)}, \ldots , z^{(m)}, \varphi)$ we have, by \eqref{carminho} and \eqref{apalachian}
\begin{multline*}
\big\Vert (T^{k-1}(\theta))_v  : (\ell_{q,\infty}^n)^d  \to \big( \ell_{q, ( \frac{1}{1-\theta} )^{k-1}}^n \big)' \big\Vert \\
 \leq C_{m,r} \Vert P \Vert_{\mathcal{P}(^{m} \ell_{r}^{n})}\| \varphi \|_{\ell_1(\Jj (m,n))'}  
\| z^{(2)} \|_{\ell_{q,\frac{1}{\theta}}} \cdots \| z^{(k)} \|_{\ell_{q,\frac{1}{\theta}}}\| z^{(k+2)} \|_{\ell_{q,\infty}} \cdots \| z^{(m)} \|_{\ell_{q,\infty}} 
\end{multline*}
and
\begin{multline*}
\big\Vert  ({T_{k+1}})_v  : (\ell_{q,1}^n)^d  \to \big( \ell_{q,\infty}^n \big)' \big\Vert \\
 \leq C_{m,r} \Vert P \Vert_{\mathcal{P}(^{m} \ell_{r}^{n})}\| \varphi \|_{\ell_1(\Jj (m,n))'}  
\| z^{(2)} \|_{\ell_{q,\frac{1}{\theta}}} \cdots \| z^{(k)} \|_{\ell_{q,\frac{1}{\theta}}}\| z^{(k+2)} \|_{\ell_{q,\infty}} \cdots \| z^{(m)} \|_{\ell_{q,\infty}} \,.
\end{multline*}
Once again, we may interpolate using Theorem~\ref{int op conos}, \eqref{int lorentz conos} and \eqref{int lorentz dual} to have
\begin{multline*}
\big\Vert  ( T^{k}(\theta)  v : (\ell_{q,\frac{1}{\theta}}^n)^d \to \big( \ell_{q,\frac{1}{(1-\theta)^k}}^n \big)' \big\Vert \\
\leq C_{m,r} \Vert P \Vert_{\mathcal{P}(^{m} \ell_{r}^{n})}\| \varphi \|_{\ell_1(\Jj (m,n))'}  
\| z^{(2)} \|_{\ell_{q,\frac{1}{\theta}}} \cdots \| z^{(k)} \|_{\ell_{q,\frac{1}{\theta}}}\| z^{(k+2)} \|_{\ell_{q,\infty}} \cdots \| z^{(m)} \|_{\ell_{q,\infty}}
\end{multline*}
for every $0 < \theta < 1$. Taking supremum as before this gives
\[
\big\Vert T^{k}(\theta) :  \ell_{q, \left( \frac{1}{1-\theta} \right)^{k}}^n \times \underbrace{(\ell_{q,\frac{1}{\theta}}^n)^d \times \dots \times (\ell_{q,\frac{1}{\theta}}^n)^d}_{k} \times \underbrace{(\ell_{q,\infty}^n)^d \times\dots \times (\ell_{q,\infty}^n)^d}_{m-k-1}  \to \ell_1(\mathcal J (m,n)) \big\Vert 
\leq C_{m,r} \Vert P \Vert_{\mathcal{P}(^{m} \ell_{r}^{n})} \,. \qedhere
\]
\end{proof}

\begin{proof}[Proof of Theorem~\ref{teoremamulti 2}.]
For $m\ge 5$, we choose $\theta = \frac{\log(m+\frac{3}{2})}{m-1+\log(m+\frac{3}{2})}$.
Then $\frac{1}{\theta}\ge \frac{m}{\log(m)}$ and
\[
\left( \frac{1}{1-\theta} \right)^{k} 
	 = \Big( 1 + \frac{\log(m+\frac{3}{2})}{m-1} \Big)^{m-2}\\
	 \ge \frac{m}{\log{m}}.
\]
Therefore $\Vert \ell_{q, \left( \frac{1}{1-\theta} \right)^{k}}^{n} \hookrightarrow \ell_{q,\frac{m}{\log(m)}}^{n} \Vert = \Vert \ell_{q,  \frac{1}{\theta} }^{n} \hookrightarrow \ell_{q,\frac{m}{\log(m)}}^{n} \| = 1$. Using Lemma~\ref{induction lemma} with $k = m-2$ the result follows.
For $m=3$ and $m=4$ just take $\theta =\frac12$ and $\theta =\frac{3}2-\frac{\sqrt{5}}{2}$ in Lemma~\ref{induction lemma}, respectively.
\end{proof}

We finish this section with some comments on the hypercontractivity  of the inclusion of $\ell_{q,s}$ in $\mon \Pp(^m \ell_r)$. For the $\ell_\infty$ case it is known (see \cite[Theorem~2.1]{bayart2017multipliers}) that the inclusion $\ell_{\frac{2m}{m-1},\infty}$ in $\mon \Pp(^m \ell_\infty)$ is hypercontractive in the sense that there exists a constant $C>0$ such for every $P\in \Pp(^m \ell_\infty)$, 
\begin{equation*}
    \sum_{\mathbf j \in \mathcal J(m,n)} \vert c_{\mathbf j}(P) z_{\mathbf j}  \vert \leq C^m \Vert z \Vert_{\ell_{\frac{2m}{m-1},\infty}}^m \Vert P \Vert_{P(^m \ell_\infty)}.
\end{equation*}

For $1<r\le 2$, although we do not know if $\ell_{q,\infty}$ lies in the set $\mon \Pp(^m \ell_r)$ it is easy to see that we cannot expect to have a hypercontractive inequality as above. 

\begin{remark}
Proceeding as in the proof of the upper inclusion in Theorem~\ref{thm: main theorem hb} (see  \eqref{desigBayart}) with $m=\lfloor \log(n+1)\rfloor$  we would have that  $$ \frac{1}{\Vert z \Vert_{\ell_{q,\log m}}\log(n+1)^{1-\frac{1}{r}}} \sum_{j = 1}^n |z_j^*| $$ is bounded independently of $n$ for every $z\in \ell_{q,\log m}$. Take now $z=(j^{-1/q}\log(j)^{-2/\log(m)})_j$. Then $\|z\|_{\ell_{q,\log m}}\le \left(\sum_{j=1}^\infty \frac1{j\log^2(j)}\right)^{\frac1{\log m}}.$ But,
\begin{multline*}
 \frac{1}{\Vert z \Vert_{\ell_{q,\log m}}\log(n+1)^{1-\frac{1}{r}}} \sum_{j = 1}^n |z_j^*|  \gg  \frac1{\log(n+1)^{1-\frac{1}{r}}} \sum_{j = 1}^n \frac1{j^{1/q}\log(j)^{\frac2{\log{m}}}}  \\  \gg  \frac{e^2}{c\log(n+1)^{1-\frac{1}{r}}}\sum_{j = 1}^n \frac1{j^{1/q}} 
 \ge \frac{e^2}{c\log(n+1)^{1-\frac{1}{r}}} n^{1/q'}q'.
 \end{multline*}
 Since $q'=mr'=\lfloor \log(n+1)\rfloor r'$, the last expression is $\gg \log(n)^{\frac1{r}}$.
This shows that there exists no constant $C>0$ such that for every $n$ and $m$ and all $P\in\Pp(^m \mathbb{C}^{n})$ we have
\begin{equation*}
    \sum_{\mathbf j \in \mathcal J(m,n)} \vert c_{\mathbf j}(P) z_{\mathbf j}  \vert \leq C^m \Vert z \Vert_{\ell_{q,\log m}}^m \Vert P \Vert_{\Pp(^m \ell_r^n)}.
\end{equation*}
\end{remark} 

On the other hand, applying carefully   the ideas developed in this section, it is possible to obtain hypercontractive inequalities in some cases.

\begin{remark}\label{rem: hyper}
Given $\varepsilon>0$, there exists a constant $C>0$ such that for every $m\ge 3$, $n\in \mathbb N$ and every $P \in \mathcal P(^m \mathbb C^n)$
 \[
     \sum_{\jj \in \Jj(m,n)}  \vert c_{\mathbf j}(P) z_{\jj} \vert \le C(1+\varepsilon)^m \Vert P \Vert_{\mathcal P(^m\ell_r^n)} \Vert z \Vert_{\ell_{q,2}^n}^{m}.
\]
To see this fix $1<r\le 2$, $m\ge 3$, and take $z$, $z^{(m-2)}$, $z^{(m-1)}$, $w\in \mathbb C^n$ such that 
 $z^{(m-1)}=z^{(m-1)*}$ and $w=w^*$. Then we have, using Lemma~\ref{lem: BDS} (see also \eqref{BDS modif}) and Lemma~\ref{lem: alpha(j)},
\begin{align*}
\sum_{\jj \in \Jj(m,n)} & | c_\jj(P) z_{j_1} \dots z_{j_{m-3}}z^{(m-2)}_{j_{m-2}}z^{(m-1)}_{j_{m-1}}w_{j_m} | 
\\    &  \le em \Vert P \Vert_{\mathcal P(^m\ell_r^n)} \sum_{\jj \in \Jj(m-1,n)} \vert  z_{j_1} \dots z_{j_{m-3}}z^{(m-2)}_{j_{m-2}}z^{(m-1)}_{j_{m-1}} \vert \cdot \Big(\frac{(m-1)^{m-1}}{\alpha(\mathbf j)^{\alpha(\mathbf j)}}\Big)^{1/r} \big(\sum_{j_m=j_{m-1}}^n w_{j_m}^r \big)^{1/r} \\
	& \le em^{3} Cm^{e^{r'-1}} \Vert P \Vert_{\mathcal P(^m\ell_r^n)} \sum_{j_{m-2}=1}^n|z^{(m-2)}_{j_{m-2}}| \Big(\sum_{\jj\in \Jj(m-3,j_{m-2})} |\jj| \vert  z_{\jj} \vert  \Big)\sum_{j_{m-1}=j_{m-2}}^n|z^{(m-1)}_{j_{m-1}} | \big(\sum_{j_m=j_{m-1}}^n w_{j_m}^r \big)^{1/r} \\
	& \le Cm^{e^{r'}}\|w\|_{\ell_{q,\infty}} \Vert P \Vert_{\mathcal P(^m\ell_r^n)} \sum_{j_{m-2}=1}^n|z^{(m-2)}_{j_{m-2}}| \Big(\sum_{l=1}^{j_{m-2}} \vert  z_{l} \vert \Big)^{m-3}\sum_{j_{m-1}=j_{m-2}}^n|z^{(m-1)}_{j_{m-1}} |j_{m-1}^{\frac1{r}-\frac1{q}}   \\
		& \le Cm^{e^{r'}}\|w\|_{\ell_{q,\infty}} \Vert P \Vert_{\mathcal P(^m\ell_r^n)} \sum_{j_{m-2}=1}^n|z^{(m-2)}_{j_{m-2}}| \Big(({j_{m-2}})^{1-\frac1{q}} \Vert  z \Vert_{\ell_{q,\infty}} \Big)^{m-3}\|z^{(m-1)}\|_{\ell_{q,\infty}}(r'+1)j_{m-2}^{\frac2{q'}-\frac1{r'}}   \\
		&  \le (r'+1)Cm^{e^{r'}}\|w\|_{\ell_{q,\infty}}\|z\|_{\ell_{q,\infty}}^{m-3}\|z^{(m-2)}\|_{\ell_{q,1}}\|z^{(m-1)} \|_{\ell_{q,\infty}} \Vert P \Vert_{\mathcal P(^m\ell_r^n)},
		\end{align*}
where in the penultimate inequality we used the bound of the identity from $\ell_1^k$ to $\ell_{q,\infty}^k$ that may be found for example in \cite[Lemma~22]{defant2006norms}.
On the other hand, we also have,
\begin{align*}
\sum_{\jj \in \Jj(m,n)} & | c_\jj(P) z_{j_1} \dots z_{j_{m-3}}z^{(m-2)}_{j_{m-2}}z^{(m-1)}_{j_{m-1}}w_{j_m} | 
\\    &  \le em \Vert P \Vert_{\mathcal P(^m\ell_r^n)} \sum_{\jj \in \Jj(m-1,n)} \vert  z_{j_1} \dots z_{j_{m-3}}z^{(m-2)}_{j_{m-2}}z^{(m-1)}_{j_{m-1}} \vert \cdot \Big(\frac{(m-1)^{m-1}}{\alpha(\mathbf j)^{\alpha(\mathbf j)}}\Big)^{1/r} \big(\sum_{j_m=j_{m-1}}^n w_{j_m}^r \big)^{1/r} \\
	& \le em^{3} Cm^{e^{r'-1}} \Vert P \Vert_{\mathcal P(^m\ell_r^n)} \sum_{j_{m-1}=1}^n|z^{(m-1)}_{j_{m-1}}| \Big(\sum_{\jj\in \Jj(m-3,j_{m-2})}|\jj| \vert  z_{\jj} \vert  \Big)\sum_{j_{m-2}=1}^{j_{m-1}}|z^{(m-2)}_{j_{m-2}} | \big(\sum_{j_m=j_{m-1}}^n w_{j_m}^r \big)^{1/r} \\
	& \le Cm^{e^{r'}}\|w\|_{\ell_{q,\infty}} \Vert P \Vert_{\mathcal P(^m\ell_r^n)} \sum_{j_{m-1}=1}^n|z^{(m-1)}_{j_{m-1}}| \Big(\sum_{l=1}^{j_{m-2}} \vert  z_{l} \vert \Big)^{m-3}\sum_{j_{m-2}=1}^{j_{m-1}}|z^{(m-2)}_{j_{m-2}} |j_{m-1}^{\frac1{r}-\frac1{q}}   \\
	& \le Cm^{e^{r'}}\|w\|_{\ell_{q,\infty}} \Vert P \Vert_{\mathcal P(^m\ell_r^n)} \sum_{j_{m-1}=1}^n|z^{(m-1)}_{j_{m-1}}| \Big(({j_{m-2}})^{1-\frac1{q}} \Vert  z \Vert_{\ell_{q,\infty}} \Big)^{m-3}j_{m-1}^{1-\frac1{q}}\|z^{(m-2)}\|_{\ell_{q,\infty}}j_{m-1}^{\frac1{r}-\frac1{q}}   \\
		&  =Cm^{e^{r'}}\|w\|_{\ell_{q,\infty}}\|z\|_{\ell_{q,\infty}}^{m-3}\|z^{(m-2)}\|_{\ell_{q,\infty}}\|z^{(m-1)}
\|_{\ell_{q,1}} \Vert P \Vert_{\mathcal P(^m\ell_r^n)}.
\end{align*}
Thus, proceeding as in Lemma~\ref{induction lemma} we may construct an operator
which is bounded from  $\ell_{q,\infty}^d$ to $\big(\ell_{q,1}\big)'$ and also from  $\ell_{q,1}^d$ to $\big(\ell_{q,\infty}\big)'$. Applying the $K$-interpolation  method restricted to the cone of non-increasing sequences to this operator we can conclude that for any $z=z^*$, 
\begin{equation*}
\sum_{\jj \in \Jj(m,n)}  | c_\jj(P) z_{\jj} | \le \sqrt{(1+r')}Cm^{e^{r'}}\|z\|_{\ell_{q,\infty}}^{m-2}\|z\|_{\ell_{q,2}}^2 \Vert P \Vert_{\mathcal P(^m\ell_r^n)}  \le C(1+\varepsilon)^m\Vert P \Vert_{\mathcal P(^m\ell_r^n)} \Vert z \Vert_{\ell_q^n}^{m}.
\end{equation*}
Therefore, by \eqref{basta la desi para z*}, we have proved our claim. \\

With some extra work it can proved, in a similar way, that given any $s\ge1$ and $\varepsilon>0$, there exist some $m_0$
and some $C>0$ such that for every $n\in \mathbb N$, all $m\ge m_0$ and every polynomial $P \in \mathcal P(^m \mathbb C^n)$ we have
 \begin{equation*}
     \sum_{\jj \in \Jj(m,n)}  \vert c_{\mathbf j}(P) z_{\jj} \vert \le C(1+\varepsilon)^m \Vert P \Vert_{\mathcal P(^m\ell_r^n)} \Vert z \Vert_{\ell_{q,s}^n}^{m}.
\end{equation*}
\end{remark}

\section{Some consequences}\label{aplicaciones}
We now provide several consequences of the results obtained in the previous sections.

\subsection{Mixed unconditionality for spaces of $\boldsymbol{m}$-homogeneous polynomials}
Let us recall that, if  $(P_i)_{i \in \Lambda}$ is a Schauder basis of $\Pp(^m \CC^n)$, for $ 1 \leq  r,s \leq \infty$ and $ n,m \in \NN$, the mixed unconditional basis constant $ \chi_{r,s}((P_i)_{i \in \Lambda})$ is defined as the best constant $C > 0$ such that 
\[
\| \sum_{ i \in \Lambda} \theta_i c_i  P_i \|_{\mathcal{P}(^m \ell_s^n)} \leq C \| \sum_{ i \in \Lambda}  c_i P_i \|_{\mathcal{P}(^m \ell_r^n)},
\]
for every $P = \displaystyle\sum_{ i \in \Lambda}  c_i P_i \in \mathcal{P}(^m \mathbb{C}^n)$ and every choice of complex numbers $(\theta_i)_{i \in \Lambda}$ of modulus one. Once we have this, the $(r,s)$-mixed unconditional constant of $\mathcal{P}(^m \mathbb{C}^n)$ is defined as
\[
\chi_{r,s}(\mathcal{P}(^m \mathbb{C}^n)) := \inf\{\chi_{p,q}((P_i)_{i \in \Lambda}) :  (P_i)_{i \in \Lambda} \text{ basis for } \mathcal{P}(^m \mathbb{C}^n)\}.
\]
This notion was introduced by Defant, Maestre and Prengel in \cite[Section 5]{defant2009domains}.\\

In \cite{galicer2016sup} the exact asymptotic asymptotic growth of the mixed-$(r,s)$ unconditional constant as $n$ tends to infinity was computed for many values of $p$ and $q$'s.
To achieve this the authors  proved that 
\[
\chi_{r,s}(\mathcal{P}(^m \mathbb{C}^n)) \sim \chi_{r,s}\big((z_{\mathbf j})_{\mathbf j \in \mathcal J(m,n)} \big).
\]
We complete the result  given in \cite[Theorem~ 3.4]{galicer2016sup} by providing the exact asymptotic asymptotic growth for the remaining cases. In this way we have the behaviour of $\chi_{p,q}(\mathcal{P}(^m \mathbb{C}^n))$ for \emph{every} $1 \leq p, q \leq \infty$.

\begin{theorem}\label{cte incond mon}
For each $m \in \mathbb N$ we have
\[
\begin{cases}  \; \chi_{r,s}(\mathcal{P}(^m \mathbb{C}^n)) \sim 1 & \text{ for } (I): \; [\frac{1}{r} + \frac{m-1}{2m} \leq \frac{1}{s} \wedge \frac{1}{r} \leq \frac{1}{2} ] \text{ or } [ \frac{m-1}{m} + \frac{1}{mr} < \frac{1}{s} \wedge \frac{1}{2} \leq \frac{1}{r} ], \\
\;\chi_{r,s}(\mathcal{P}(^m \mathbb{C}^n)) \sim n^{m (\frac{1}{r}-\frac{1}{s}+\frac{1}{2}) - \frac{1}{2}} & \text{ for } (II) \; \; [ \frac{1}{r} + \frac{m-1}{2m} \geq \frac{1}{s} \wedge \frac{1}{r} \leq \frac{1}{2}  ] ,\\
\; \chi_{r,s}(\mathcal{P}(^m \mathbb{C}^n)) \sim n^{(m-1)(1-\frac{1}{s}) + \frac{1}{r}  -\frac{1}{s}} & \text{ for } (III) \; : \; [ 1 - \frac{1}{m} + \frac{1}{mr} \geq \frac{1}{s} \; \wedge \; \frac{1}{2} < \frac{1}{r} <1 ]. \\
\end{cases}
\]
\end{theorem}

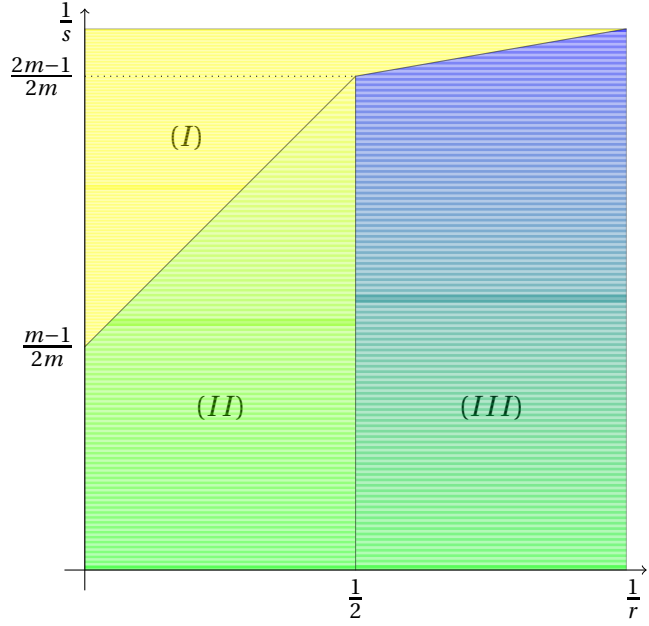
\begin{figure} \label{Figura incondicionalidad}
\begin{center}
\begin{tikzpicture}[scale=0.9]

	\draw (1.5,6.4) node {$(I)$};
  \path[draw, shade, top color=yellow, bottom color=yellow, opacity=.3]
     (4,7.3) node[below] {$ $}  -- (8, 8) -- (0, 8) -- (0,3.3) -- cycle;

    \draw (2,2.4) node {$(II)$};
  \path[draw, shade, top color=yellow, bottom color=green, opacity=.3]
     (0,0) node[below] {$ $}  -- (0, 3.3) -- (4,7.3) -- (4,0) -- cycle;

	\draw (6,2.4) node {$(III)$};
  \path[draw, shade, top color=blue, bottom color=green, opacity=.3]
    (4,7.3) node[below] {$ $} -- (8,8) -- (8,0) node[below] {$ $}
     -- (4, 0) -- cycle;


\draw[dotted] (0,7.3) -- (4,7.3); 

\draw (4,0) node[below] {$\frac12$};
\draw (0,7.3) node[left] {$\frac{2m-1}{2m}$};
\draw (0,3.3) node[left] {$\frac{m-1}{2m}$};

\draw (8.1, 0) node[below] {$\frac1{r} $};
  \draw[->] (-0.3,0) -- (8.3, 0);
\draw (0,8.1) node[left] {$\frac1{s} $};
  \draw[->] (0,-0.3) -- (0, 8.3);
 \end{tikzpicture}
\end{center}

\caption{Graphical overview of the mixed unconditional constant described in Theorem~\ref{cte incond mon}.}

\end{figure}

\begin{proof}
The behaviour of $\chi_{r,s}(\mathcal{P}(^m \mathbb{C}^n))$ in regions $(I)$ and $(II)$ was already given in \cite[Theorem~3.4]{galicer2016sup}. We now deal with $(III)$. By
Theorem~\ref{teo ellq} we know that $\ell_{q} \subset \mon \Pp(^m \ell_r).$ Thus, for every polynomial $P(z)= \sum_{\vert \alpha\vert = m} c_{\alpha} z^{\alpha} \in \Pp(^m \mathbb{C}^{n})$ we have
\begin{equation}\label{desigualdad por el mon}
\sum_{\vert \alpha\vert = m} \vert c_{\alpha} z^{\alpha} \vert \leq C^m \Vert z \Vert_{\ell_{q}}^m \Vert P \Vert_{\Pp(^m \ell_r)},     
\end{equation}
where  $q:=(mr')'$. Since  
\begin{equation} \label{comparacion de las normas}
 \Vert z \Vert_{\ell_{q}} \leq n^{\frac{1}{q}-\frac{1}{s}}\Vert z \Vert_{\ell_{s}}    ,
\end{equation}
combining \eqref{desigualdad por el mon} and \eqref{comparacion de las normas} yields
\begin{align}
\chi_{r,s}(\mathcal{P}(^m \mathbb{C}^n)) \leq  n^{(\frac{1}{q}-\frac{1}{s})m} = n^{(1-\frac{1}{mr'}-\frac{1}{s})m} =  n^{m(1-\frac{1}{s})-\frac{1}{r'}}=n^{(m-1)(1-\frac{1}{s}) + \frac{1}{r}  -\frac{1}{s}}.
\end{align}
\end{proof}
\subsection{Mixed Bohr radius}

Let $K(B_{\ell_p^n},B_{\ell_q^n}) $ be the $n$-dimensional $(p,q)$-Bohr radius for holomorphic functions on $\CC^n$. That is, $K(B_{\ell_p^n},B_{\ell_q^n}) $ denotes the greatest number $r\geq 0$ such that for every entire function $f(z)=\sum_{\alpha} a_{\alpha} z^{\alpha}$ in $n$-complex variables, we have the following (mixed) Bohr-type  inequality
\[
\sup_{z \in  r \cdot B_{\ell_q^n}} \sum_{\alpha} \vert a_{\alpha} z^{\alpha} \vert \leq \sup_{z \in  B_{\ell_p^n}} \vert f(z) \vert.
\]
The exact asymptotic growth of $K(B_{\ell_p^n},B_{\ell_q^n}) $ with $n$ was given in \cite[Theorem~1.2]{galicer2017mixed}. More precisely, $K(B_{\ell_p^n},B_{\ell_1^n})  \sim 1$ for  every $1 \le p \le \infty$, and for $1 \leq p, q \leq \infty$, with $q \neq 1$, 
\[
K(B_{\ell_p^n},B_{\ell_q^n})  \sim \begin{cases}
		1 & \text{ if (I): } 2 \leq p \leq \infty \; \wedge \; \frac{1}{2} + \frac{1}{p} \le \frac{1}{q},\\
		\frac{\sqrt{\log(n)}} {n^{\frac{1}{2} + \frac{1}{p} - \frac{1}{q}}} & \text{ if (II): } 2 \leq p \leq \infty \; \wedge \; \frac{1}{2} + \frac{1}{p} > \frac{1}{q}, \\
		\frac{\log(n)^{1 - \frac{1}{p}}}{n^{1-\frac{1}{q}}}  & \text{ if (III): } 1 \leq p  \leq 2. \\
		\end{cases}
\]

As a consequence of our result we can give an alternative proof of the lower bounds for $K(B_{\ell_p^n},B_{\ell_q^n}) $ for the case $1 \leq p \leq 2$ (and every $1 \leq q \leq \infty$). It should be noted that this is the most complicated part of  \cite[Theorem~1.2]{galicer2017mixed}. Let us see how.\\

By \cite[Theorem~5.1]{defant2009domains} and Lemma~\ref{lem: 1er result H_infty}, there is a constant $C:=C(p)>0$ such that for every polynomial $P$ in $n$ complex variables we have
\begin{equation}\label{desig radio de bohr}
\sum_{\mathbf j\in\mathcal J(m,n)}  \vert c_{\mathbf j}(P) z_{\mathbf j} \vert  \leq C^m \Vert z \Vert_{(m_{\Psi_p})_n}^m \Vert P \Vert_{\mathcal P(^m \ell_p^n)},
\end{equation}
where $ (m_{\Psi_p})_n $ is defined as the quotient space induced by the mapping
\begin{align*}
\pi_n : m_{\Psi_p} &\rightarrow \CC^n \\
		x & \mapsto (x_1, \ldots, x_n).
\end{align*}
Note that there is a constant $D=D(p,q)>0$ such that $\Vert z \Vert_{(m_{\Psi_p})_n} \leq D \frac{n^{1-\frac{1}{q}}}{\log(n)^{1-\frac{1}{p}}} \Vert z \Vert_{\ell_q^n}.$ Therefore, by \eqref{desig radio de bohr} we have
 \[
\sum_{\mathbf j\in\mathcal J(m,n)}  \vert c_{\mathbf j}(P) z_{\mathbf j} \vert  \leq (CD)^m \left(\frac{n^{1-\frac{1}{q}}}{\log(n)^{1-\frac{1}{p}}} \right)^m  \Vert z \Vert_{\ell_q^n}^m \Vert P \Vert_{\mathcal P(^m \ell_p^n)},
\]
This implies that $\chi_{p,q}(\mathcal{P}(^m \mathbb{C}^n))^{1/m} \ll \frac{n^{1-\frac{1}{q}}}{\log(n)^{1-\frac{1}{p}}}$. It should be noted that here is important to have  control of the growth of the $(p,q)$-mixed unconditional constant also in terms of $m$  (the homogeneity degree), contrary to problem treated in the previous subsection. ~
The result now follows using that (see \cite[Lemma~2.2.]{galicer2017mixed}) for every $n \in \zN$ and $1 \le p , q \le \infty$ we have
		\[
		 K(B_{\ell_p^n},B_{\ell_q^n})  \sim \frac{1}{\sup_{ m \ge 1}  \chi_{p,q}(\mathcal{P}(^m \mathbb{C}^n))^{1/m}}.
		\]

\subsection{Multipliers}

A sequence $(a_n)_{n \in \mathbb{N}}$ is a multiplier for $\mon \mathcal{P}(^m \ell_r)$ if
$$
(a_n)_{n \in \mathbb N} \cdot \ell_r \subset \mon \mathcal{P}(^m \ell_r),
$$
where the product $(a_n)_{n \in \mathbb N} \cdot \ell_r$ is just the coordinate-wise multiplication. Let  $p=(p_1,p_2, \dots)$ be the sequence of the prime numbers.
It is well-known  that for $r \ge 2$, the sequence $\frac{1}{
			p^{\frac{m-1}{2m}}
		}$ is a multiplier for  $\mon \mathcal{P}(^m \ell_r)$ (this can be as an immediate consequence of  \cite[Theorem~5.1 (3) ]{bayart2016monomial}).

For  $1 <r< 2$ in \cite[Theorem~5.3.]{bayart2016monomial} prove this  up to an $\varepsilon$, showing that for each $m$ and every $\veps>\frac 1r$
\begin{equation}\label{multiplier primo BDS}
\frac{1}{p^{\sigma_m}\big(\log(p)\big)^{\veps}}\cdot \ell_r\subset \mon\mathcal P(^m\ell_r),    
\end{equation}
where  $\sigma_m=\frac{m-1}{m}\left(1-\frac 1r\right)$.
As a consequence of our results, we can improve this, showing that, for $1 < r \le 2 $, even the sequence $(\frac{1}{n^{\sigma_m}})_{n \in \mathbb N}$ is a multiplier for $\mon \mathcal{P}(^m \ell_r)$. 

\begin{theorem}
For $1 <r< 2$ and  $m\geq 3$ put $\sigma_m=\frac{m-1}{m}\left(1-\frac 1r\right)$. Then, 
\[
\big(\frac{1}{n^{\sigma_m}}\big)_n\cdot \ell_r\subset \mon\mathcal P(^m\ell_r),
\]
and $\sigma_{m}$ is best possible.
\end{theorem}

\begin{proof}
As a consequence of Theorem~\ref{teo: main poli} we know that $\ell_{q,r} \subset \mon\mathcal P(^m\ell_r)$, thus 
to prove the result it is sufficient to see that if $z \in \ell_r$ then, $\big(\frac{1}{n^{\sigma_m}}\big)_n \cdot z \in \ell_{q,r}.$ 
Suppose that $z\in  \ell_r$ is an arbitrary element (not necessarily equal to $z^*$).
Since $r>q$ we know that the norm $\Vert \cdot \Vert_{\ell_{q,r}}$ is equivalent to the following maximal norm (see \cite[Lemma~4.5]{bennett1988interpolation})
\[
\Vert w \Vert_{\ell_{q,r}}^* = \left( \sum_{n=1}^\infty n^{\frac{r}{q} -1} \left( \frac{1}{n} \sum_{k=1}^n w_k^* \right)^r  \right)^{1/r}.
\]
Then, if $w= \left(\frac{z_n}{n^{\sigma_m}}\right)_n$, by the Hardy-Littlewood rearrangement inequality (Lemma~\ref{reordenamiento}) it is easy to see that  
\[
\sum_{k=1}^n w^*_k \leq \sum_{k=1}^n z_k^* \frac{1}{k^\sigma}
\]
for every $n \in \mathbb N$. Then
\begin{multline*}
\left\Vert \left(\frac{z_n}{n^{\sigma_m}}\right)_n \right\Vert_{\ell_{q,r}} \sim \left\Vert\left(\frac{z_n}{n^{\sigma_m}}\right)_n \right\Vert_{\ell_{q,r}}^* \leq  \left( \sum_{n=1}^\infty n^{\frac{r}{q} -1} \left(\frac{1}{n} \sum_{k=1}^n z_k^* \frac{1}{k^\sigma} \right)^r  \right)^{1/r} \\
 = \left\Vert \left(\frac{z_n^*}{n^{\sigma_m}}\right)_n \right\Vert_{\ell_{q,r}}^* \sim \left\Vert\left(\frac{z_n^*}{n^{\sigma_m}}\right)_n \right\Vert_{\ell_{q,r}} 
 = \left(\sum_{n=1}^\infty \left((\frac{z_n^*}{n^{\sigma_m}})^* n^{\frac{1}{q}-\frac{1}{r}} \right)^r\right)^{1/r}  = \Vert z \Vert_{\ell_r} < \infty,
\end{multline*}
where, in the last equality, we have used the fact that $\sigma_m = \frac{1}{q}-\frac{1}{r}$.\\
To see that the exponent is optimal take, as always, $q=(mr')'$. Now, if $(z_n)_n=\left(\frac{1}{n^{1/r} \log(n+1)^{2/r}}\right)_n \in \ell_r$ for every $\varepsilon>0$ it is easy to check that the sequence $\big( \frac{z_n}{n^{\sigma_m-\varepsilon}} \big)_n \notin \ell_{q,\infty} \supset \mon \Pp (^m\ell_r).$
\end{proof}

For $m = 2$ we cannot show that the sequence is $\big(\frac{1}{n^{\sigma_2}}\big)_n$ is a multiplier for $\mon\mathcal P(^2\ell_r)$ but using the fact that $\ell_q \subset \mon\mathcal P(^2\ell_r) $, Theorem~\ref{teo: main poli}, it is easy to see that we have the inclusion  
\begin{equation*}
\frac{1}{p^{\sigma_2}\big(\log(p)\big)^{\veps}}\cdot \ell_r\subset \mon\mathcal P \left( ^2\ell_r \right),    
\end{equation*}
for \emph{every} $\varepsilon>0$ extending \cite[Theorem~5.3.]{bayart2016monomial}  (see also \eqref{multiplier primo BDS}). We leave the details for the reader.

\newcommand{\etalchar}[1]{$^{#1}$}

\end{document}